\documentclass[11pt]{amsart}

\usepackage{microtype}

\usepackage[T1]{fontenc}

\usepackage{RNdefn}%, RNams}
\usepackage{eucal, mathtools}
\usepackage{tikz}
\usepackage{tikz-cd}
\usepackage{aliascnt}

\usepackage[nameinlink]{cleveref}
\crefname{subsection}{\S\!\!}{subsections}
\crefname{section}{\S\!\!}{\S\!}
\renewcommand{\autoref}{\Cref}

\newcommand{\St}{\mathbf{St}}
\newcommand{\LS}{\mathrm{LS}}
\newcommand{\oSt}{\mathrm{St}}
\newcommand{\Frac}{\mathrm{Frac}}
\newcommand{\VB}{\mathrm{VB}}

\newcommand{\pVB}{\mathrm{VB}_{\bF_p}}

\newcommand{\vcd}{\mathrm{vcd}}
\newcommand{\PBW}{\mathrm{PBW}}

\providecommand{\Fq}{\ensuremath\mathbb F_q}
\providecommand{\Fp}{\ensuremath\mathbb F_p}

\renewcommand{\bZ}{\mathbb{Z}}
\renewcommand{\bQ}{\mathbb{Q}}
\renewcommand{\bF}{\mathbb{F}}
\renewcommand{\rH}{\mathrm{H}}

\renewcommand{\ro}{\ensuremath\omega}

\newcommand{\m}{\to}
\newcommand{\Z}{\bZ}
\newcommand{\Q}{\bQ}

\renewcommand{\rH}{\mathrm{H}}

\numberwithin{thmcounter}{section}
\newaliascnt{thmauto}{thmcounter}

\newaliascnt{Defauto}{thmcounter}

\newaliascnt{exauto}{thmcounter}

\newaliascnt{exsauto}{thmcounter}

\newaliascnt{lemauto}{thmcounter}

\newaliascnt{propauto}{thmcounter}

\newaliascnt{corauto}{thmcounter}

\newaliascnt{conjauto}{thmcounter}

\newaliascnt{queauto}{thmcounter}

\newaliascnt{remauto}{thmcounter}

\usepackage[T1]{fontenc}
\usepackage{lmodern}

\theoremstyle{plain}
\newtheorem{theorem}[thmauto]{Theorem}
\newtheorem{definition}[Defauto]{Definition}
\newtheorem{ex}[exauto]{Example}

\newtheorem{lemma}[lemauto]{Lemma}
\newtheorem{proposition}[propauto]{Proposition}
\newtheorem{corollary}[corauto]{Corollary}
\newtheorem{conjecture}[conjauto]{Conjecture}
\newtheorem{question}[queauto]{Question}
\newtheorem{remark}[remauto]{Remark}

\newtheorem*{rem*}{Remark}
\newtheorem*{thm*}{Theorem}
\newtheorem*{exs*}{Examples}

\newcommand{\FI}{\mathrm{FI}}
\newcommand{\VIC}{\mathrm{VIC}}
\newcommand{\SI}{\mathrm{SI}}

\providecommand{\triv}{\ensuremath\mathrm{triv}}

\newtheorem*{definition*}{Definition}

\DeclareMathOperator{\Ext}{Ext}

\newcommand{\coloneq}{\mathrel{\mathop:}\mkern-1.2mu=}

	\title{Stability in the high-dimensional cohomology of congruence subgroups}

\author{Jeremy Miller}
\thanks{Jeremy Miller was supported in part by NSF grant DMS-1709726.}
\address{150 N University Street
	West Lafayette, IN-47904
	USA, Department of Mathematics, Purdue University, USA}
\email{jeremykmiller@purdue.edu}
\author{Rohit Nagpal}
\address{530 Church St, Ann Arbor, MI 48109, Department of Mathematics, University of Michigan, USA}
\email{rohitna@umich.edu}
\author{Peter Patzt}
\address{150 N University Street
	West Lafayette, IN-47904
	USA, Department of Mathematics, Purdue University, USA}
\email{ppatzt@purdue.edu}
\date{\today}



\begin{document}	
	
\begin{abstract}	
	We prove a representation stability result for the codimension-one cohomology of the level three congruence subgroup of $\SL_n(\bZ)$. This is a special case of a question of Church--Farb--Putman which we make more precise. Our methods involve proving finiteness properties of the Steinberg module for the group $\SL_n(K)$ for $K$ a field. This also lets us give a new proof of Ash--Putman--Sam's homological vanishing theorem for the Steinberg module. We also prove an integral refinement of Church--Putman's homological vanishing theorem for the Steinberg module for the group $\SL_n(\bZ)$.
	
\end{abstract}
	
\maketitle

\tableofcontents

\section{Introduction}
The purpose of this paper is to explore finiteness properties of Steinberg modules of special linear groups and applications of these finiteness properties to group (co)homology. In particular, for $K$ a field, we prove that the Steinberg modules of $\SL_n(K)$ exhibit a derived form of representation stability (see \autoref{intro:tor-bound}). The primary application of this result is \autoref{thm:main} concerning the high degree cohomology of congruence subgroups of $\SL_n(\Z)$. We also apply our results to reprove a result of Ash--Putman--Sam \cite{APS} and strengthen a result of Church--Putman \cite{CP}.  

\subsection{Cohomology of congruence subgroups}
The (co)homology of arithmetic groups is a rich subject that has had many applications in number theory and algebraic K-theory. In this paper, we focus on congruence subgroups of $\SL_n(\bZ)$. Let $\Gamma_n(p)$ denote the kernel of the reduction mod $p$ map \[\SL_n(\bZ) \to \SL_n(\bF_p).\] For $i$ small compared to $n$, which we shall refer to as the low-dimensional case, the rational cohomology groups $\rH^i(\Gamma_n(p);\bQ)$ are completely known by the work of Borel \cite{Bor}. From this calculation, one sees that congruence subgroups exhibit rational homological stability. 

Our work focuses on integral (co)homology, so all homology and cohomology groups will have integer coefficients unless otherwise specified. The torsion in the integral (co)homology of $\Gamma_n(p)$ is quite complicated even in the low-dimensional case and there are hardly any explicit calculations. For $n$ sufficiently large, the groups $\rH_1(\Gamma_n(p))$ and $\rH_2(\Gamma_n(p))$ were respectively computed by Lee--Szczarba \cite{LS} and F. Calegari \cite{calegari}. Even $\rH_3(\Gamma_n(p))$ is currently unknown for large $n$. There are however homological and representational stability results for the torsion in the low-dimensional homology of congruence subgroups \cite{Char, PU, CEFN,CE,PS,MPW,CMNR,GLcong}.

%These stability results were in fact necessary for Calegari's calculation of the second homology group. 

%\rohit{Is the claim in the commented out line above correct? Was it used in a non-trivial way? There is at least one more proof of Calegari's theorem known which does not seem to use representation stability so I am uncomfortable with using this sentence.} 

%Many integral calculations and stability results are known in the low-dimensional homology case (instead of cohomology);   see \cite{calegari, LS, Pu, CEFN}. \rohit{May be this last sentence should be removed.}

In contrast to this complete calculation of the stable rational cohomology and stability results for the torsion, very little is known about  the  cohomology outside of the stable range, even rationally. Let $p$ be an odd prime. The virtual cohomological dimension, denoted $\vcd$, of $\Gamma_n(p)$ is known to be $\binom{n}{2}$. Since the cohomology groups $\rH^{i}(\Gamma_n(p))$ vanish if $i > \vcd$, we shall refer to $\rH^{\vcd -i}(\Gamma_n(p))$ as the codimension $i$ cohomology. When $i$ is small  compared to $n$, we informally call this the high-dimensional case. This case is even more mysterious than the low-dimensional case and the only calculations known (rational or integral), for a general $n$, exist only in codimension-zero and levels $2$, $3$, and $5$ (Lee--Szczarba \cite{LS} and \cite{MPP}). In particular, Lee--Szczarba \cite{LS}) made the following computation:  \begin{align}\rH^{\vcd}(\Gamma_n(p)) \cong \oSt_n(\bF_p)\text{ for } n \geq 3 \text{ and } p = 3. \label{eq:LS-calc} \end{align} Here $\oSt_n(R)$ denotes the Steinberg module of $\SL_n(R)$. In a series of papers, Ash--Gunnells--McConnell have calculated the codimension-one cohomology of certain finite index subgroups of $\SL_4(\bZ)$; see \cite{AGM1, AGM2, AGM3}.

%Lee-Szczarba \cite[Theorem~1.5]{LS} also gave a calculation for the codimension-one cohomology of $\Gamma_3(3)$. In a series of paper, Ash-Gunnells-McConnell have calculated codimension-one cohomology of certain finite index subgroups of $\SL_n(\bZ)$ for $n \le 4$; see \cite{AGM1, AGM2, AGM3}.

Since $\oSt_n(\bF_p)$ is a free abelian group of rank $p^{\binom{n}{2}}$, the calculation  \eqref{eq:LS-calc} shows that the  codimension-zero cohomology does not stabilize in the classical sense. However, for $n \geq 3$, it does admit a uniform description independent of $n$ -- it is the Steinberg module for each $n$. This is the key feature of {\bf representation stability}.

There are many different notions of representation stability with the most basic being {\bf finite generation degree}. Let \[G_0 \hookrightarrow G_1 \hookrightarrow G_2 \hookrightarrow \ldots \] be a sequence of groups and let \[M_0 \to M_1 \to M_2 \to \ldots \] be a sequence with $M_n$ a $G_n$-representation and with $M_{n} \to M_{n+1}$ a $G_n$-equivariant map. We say $\{ M_n \}_{n \ge 0}$ has generation degree $\leq d$ if \[\Ind_{G_{n}}^{G_{n+1}} M_{n} \to M_{n+1}\] is surjective for $n \geq d$.  This is the definition of representation stability that we will use.

Note that $\rH^\ast(\Gamma_n(p))$ has a natural $\GL_n^{\pm}(\bF_p)$ action. Here the superscript $\pm$ means we restrict to matrices with determinant equal to $\pm 1$. Using Borel--Serre duality \cite{BoSe}, one can construct equivariant maps \[\rH^{\vcd-i}(\Gamma_n(p))  \to \rH^{\vcd-i}(\Gamma_{n+1}(p))  \] and thus one can make sense of the generation degree of the codimension-$i$ cohomology of congruence subgroups. Our main theorem is the following.

\begin{theorem}
	\label{thm:main}
	The sequence $\left \{\rH^{\vcd-1}(\Gamma_n(3)) \right \}_{n \ge 0}$ has generation degree $\leq 4$. In other words, the $\GL_n(\bF_3)$-equivariant map \[ \Ind_{\GL_4(\bF_3)}^{\GL_n(\bF_3)} \rH^{\binom{4}{2} -1}(\Gamma_4(3)) \to \rH^{\binom{n}{2} -1}(\Gamma_n(3)) \] is surjective for $n \ge 4$.
\end{theorem}

In \autoref{otherPrimes}, we also show that $\left \{\rH^{\vcd}(\Gamma_n(p)) \right \}_{n \ge 0}$ has generation degree equal to $ 0$ for all $p$, although this follows fairly quickly from known results. Finite generation degree allows one to control how fast the dimensions grow. In particular, \autoref{thm:main} has the following corollary.

\begin{corollary}\label{mainCor} Let $\bk$ be a field. Then for $n \geq 4$, we have that
	\[\dim_{\bk} \rH^{\vcd-1}(\Gamma_n(3);\bk) \le \frac{3^{\binom{n-4}{2}}| \GL_n(\bF_3)| }{| \GL_{n-4}(\bF_3)||\GL_4(\bF_3)|} \dim_{\bk} \rH^5(\Gamma_4(3);  \bk)  \le \frac{3^{\binom{n-4}{2}}| \GL_n(\bF_3)| }{| \GL_{n-4}(\bF_3)||\GL_4(\bF_3)|}227340. \]
\end{corollary}

%\begin{remark}
	%\label{rem:rough-estimate}
   %On our request, Mark McConnell is currently running a computer computation for $\rH^5(\Gamma_4(3);  \bk)$ which is likely to take several months to complete. The code for this calculation is based on an upcoming paper of Ash--Gunnells--McConnell \cite{AGMinprep}.  We obtain a very rough upper bound $227340$ on $\dim_{\bk} \rH^5(\Gamma_4(3);  \bk)$. The current paper will be updated to include better bounds when the computation finishes.
%\end{remark}

\begin{remark} \label{notzero}
	We note that it follows from the work of Ash in \cite{Ash} that the maps \[\rH^{\vcd-i}(\Gamma_n(p)) \m \rH^{\vcd-i}(\Gamma_{n+1}(p))\] are injective. Since $vcd \left( \Gamma_2(p) \right)=1$, we have: \[ \rH^{\vcd-1}(\Gamma_2(p)) = \rH^{0}(\Gamma_2(p) ) \cong \Z \] and so $\rH^{\vcd-1}(\Gamma_n(p))$ contains a subgroup isomorphic to $\Z$ for all $n \geq 2$. By the work of Lee-Szczarba \cite[Theorem 1.5]{LS}, \[\rH^{\vcd-1}(\Gamma_3(3)) \cong \Z^{26} \oplus (\Z/3)^8\] and so $\rH^{\vcd-1}(\Gamma_n(3))$ contains a subgroup isomorphic to $ \Z^{26} \oplus (\Z/3)^8$ for all $n \geq 3$. In particular, \autoref{thm:main} concerns nonzero groups. 
\end{remark}

%Our stability result for the high-dimensional cohomology parallels stability results for the low-dimensional cohomology \cite{PU, CEFN,CE,PS,MPW,CMNR,GLcong}. These low-dimensional stability results we necessary for Calegari's calculation of the second homology group. We hope that our stability results in the high-dimensional case will also aid in explicit calculations. 

%\autoref{thm:main} is the shadow of a more algebraic and categorical result. The sequence 
%$\left \{\rH^{\vcd-1}(\Gamma_n(3)) \right \}_{n \ge 0}$ has the structure of a module over a monoid that we call the {\bf apartment monoid} (see \autoref{secAp}). In \autoref{maintheoremfancy}, we reformulate the main result in this language. This reformulation is essential to our proof method. 

%\rohit{Uncomment the paragraph above to  add apartment monoid in the introduction. This will require additional work and I vote against apartment monoid in the introduction.}

\autoref{thm:main} can be viewed as a special case of a conjecture of Church--Farb--Putman \cite{CFPconj}. We pose a refined version of this conjecture in \autoref{sec:stability-conjectures}.

Although \autoref{thm:main} appears to be the first representation stability result concerning cohomology groups near the virtual cohomological dimension, we expect similar patterns to exist more generally. For example, we would be interested in knowing if representation stability results hold in other contexts such as congruence subgroups of mapping class groups or pure braid groups.

\subsection{Finiteness properties of the Steinberg module}

%Many of the known calculations for high-dimensional cohomology of $\Gamma_n(p)$ for odd primes $p$ make use of a $\Gamma_n(p)$-free resolution of the Steinberg module $\oSt_n(\bQ)=\oSt_n(\Z)$ due to Lee--Szczarba (or its variation due to Ash--Gunnells--McConnell \cite{AGM}). This resolution is called the Sharbly resolution (after Lee--Szczarba) and we recall its construction in detail in \autoref{sec:intro-sharbly} \rohit{Peter told me that the name Sharbly is after Lee--Szczarba}\jeremy{That strikes me as weird}. Roughly speaking, the Sharbly resolution for $\oSt_n(K)$ for $K$ a field has relations that are unbounded in $n$. This poses a major difficulty in proving representation stability results like the one in \autoref{thm:main}.

As mentioned earlier, our results on the cohomology of congruence subgroups are powered by stability results for Steinberg modules. To state these results, we need to recall a certain categorical framework. Let $\prod_{n \ge 0} \Mod_{\bZ[\GL_n(R)]}$ denote the product of the categories of $\bZ[\GL_n(R)]$-modules. There is a natural symmetric monoidal structure on this category often called the convolution product given by the formula \[ (M \otimes N)_n \coloneq \bigoplus_{i + j = n} \Ind_{\GL_{i}(R) \times \GL_{j}(R)}^{\GL_{n}(R)} M_i \otimes_{\bZ} N_j \] where  $M = \{M_n \}_{n \ge 0}$ and $N = \{ N_n \}_{n \ge 0}$ are two sequences of representations in  $\prod_{n \ge 0} \Mod_{\bZ[\GL_n(R)]}$. This product is the same as the Day convolution product induced by the direct sum monoidal structure on the groupoid of finite-rank free $R$-modules. Let $\oSt$ denote the sequence $\{\oSt_n(R)\}_{n \geq 0}$. Since $\GL_{n}(R)$ acts  on $\oSt_n(R)$, we can regard $\oSt$ as an object in the category $\prod_{n \ge 0} \Mod_{\bZ[\GL_n(R)]}$. In fact, $\oSt$ is a monoid object in this category with respect to this monoidal structure and for this reason we call $\oSt$ the {\bf Steinberg monoid}. One can extend the classical notion of a Koszul algebra from commutative algebra to  the general setting of monoidal categories. We show that $\oSt$ is Koszul in this sense. See \autoref{subsec:koszul-def} for further discussion of Koszulness.

\begin{theorem}
	\label{intro-koszul-main}
	 Let $K$ be a field. 
	Then $\oSt = \{\oSt_n(K) \}_{n \ge 0}$ is Koszul  in the monoidal category $(\prod_{n \ge 0} \Mod_{\bZ[\GL_n(K)]}, \otimes)$. 
\end{theorem}

Our proof involves showing that  $\oSt$ is a ``Poincar\'e--Birkhoff--Witt''-monoid which allows us to follow an argument due to Priddy (see \cite{priddy-koszul}) to prove our Koszulness result.

Let $\bA = \lw(\triv_1)$  be the exterior algebra on $\triv_1 = \{E_n \}_{n \ge 0}$ where \[E_n \coloneq \begin{cases}
0 &\mbox{if } n \neq 1\\
\bZ &\mbox{if } n = 1 \text{ (with the trivial action of } \GL_1(K)).
\end{cases}  \] By exterior algebra, we mean the free skew commutative monoid with respect to this monoidial structure. We call this exterior algebra the {\bf apartment monoid} and construct a surjective map $\bA \to \oSt$ of monoids (surjectivity of this map reflects the fact that $\oSt_n(K)$ is generated by apartment classes; see \autoref{remarkapartment} for more on this). Koszulness of exterior algebras is well-known (see the paragraphs after \autoref{thm:apartment-monoid-is-Koszul}), and so the theorem above tells us that $\bA \to \oSt$ is a map of Koszul monoids. We prove a general technical result (\autoref{prop:map-of-Koszul-algebras}) on maps of Koszul monoids which we expect to be useful in many other situations. This result applied to the map $\bA \to \oSt$ proves the following finiteness theorem which is the main technical result of our paper.

\begin{theorem}
	\label{intro:tor-bound}
	For a field $K$, then $\Tor^{\bA}_{i}(\oSt, \bZ)$  (regarded as a graded abelian group) is supported in degrees $\le 2i$.
\end{theorem}

%The theorem above is equivalent to proving a finiteness property of the Sharbly resolution (as in \cite{AGM}) of the Steinberg module. This is formalized in \autoref{thm:finiteness-property-Sharbly}. 

%  Our method also shows that a similar finiteness result holds for a resolution due to Lee--Szczarba \cite[Theorem~3.1]{LS}.

%Applying this technical result to the Koszul monoid in the theorem above, we prove a finiteness property of the Sharbly resolution (as in \cite{AGM}) of the Steinberg module. This finiteness property roughly amounts to bounding highes syzygies of the Steinberg monoid as a module over an exterior algebra in the monoidal category $\prod_{n \ge 0} \Mod_{\bZ[\GL_n(K)]}$ (see Theorem~\ref{thm:shably-main} for a precise statement).  

%\rohit{I think \autoref{prop:map-of-Koszul-algebras} is a more conceptual version of part 1 of Theorem 18.1 of the GKRW $E_n$-algebra paper.}

%This theorem will allow us to bound the degrees of terms in free resolutions of the Steinberg monoid (see \autoref{cor:existence-of-resolution}) as a module over apartment monoid. This can be thought of as a representation stability result for the Steinberg monoid.   \jeremy{Make this paragraph better. }

%\jeremy{There are currently no definition environments in this paper. That is weird.}
%

%

Let $\bA$ be the apartment monoid defined above for $K = \bF_3$. The sequence $\{\rH_i(\Gamma_n(3); \oSt_n(\bZ))\}_{n \ge 0}$ forms an $\bA$-module which we call $\rH_i(\Gamma(3);\oSt(\bZ))$. Using Borel--Serre duality,  \autoref{thm:main} can be rephrased as saying that \[\Tor_0^{\bA}(\rH_1(\Gamma(3);\oSt(\bZ)), \bZ)\] is supported in degree $\le 4$ (see \autoref{maintheoremfancy}). The monoid $\bA$ plays a similar role in this paper as $\FI$ plays in Church--Ellenberg--Farb \cite{CEF} and $\VIC$ and $\SI$ play in Putman--Sam \cite{PS}; the category of $\bA$-modules is the category governing representation stability in this context. To prove \autoref{thm:main}, we will only need 	 \autoref{intro:tor-bound} (and hence \autoref{intro-koszul-main}) for $K=\bF_3$. However, Koszulness of the Steinberg monoid in the general case lets us provide a new proof of the following theorem due to Ash--Putman--Sam.

\begin{theorem}[\cite{APS}]
	\label{intro-APS}
$\rH_i(\GL_n(K); \oSt_n(K)) = 0$ for all $n \ge 2 i + 2$ where $K$ is a field and $\oSt_n(K)$ is the Steinberg module for the group $\SL_n(K)$. 
\end{theorem}

\begin{remark}
We have recently been informed that our Koszulness result for the Steinberg monoid is the input needed for the $E_k$-algebra stability machine of Galatius--Kupers--Randal-Williams \cite{Ekalgebra} to apply to homology with coefficients in Steinberg modules. In particular, using our result they were able to improve the work of Ash-Putman-Sam and prove slope $2/3$ vanishing range for $\rH_\ast(\GL_n(K);\oSt_n(K))$ for $K$ a field. 
\end{remark}

%Interplay between Koszulness and homological stability can be seen from the following general fact due to Hepworth: Let $G_n$ be a reasonable sequence of groups and $A = \{A_n \}_{\ge 0}$ be the sequence of representations in $\prod_{n \ge 0} \Mod_{\bZ[G_n]}$ such that $A_n = \bZ$ with the trivial action of $G_n$. By \cite[Theorem~A]{HepEdge} (see \cite[Theorem~18.1]{Ekalgebra} for a related result), if $A$ is a Koszul monoid then the product map \[\rH_0(G_1; A_1) \otimes \rH_i(G_{n-1}; A_{n-1}) \to \rH_i(G_{n}; A_{n})\] is surjective for $n \ge 2i+1$ and an isomorphism for $n \ge 2i + 2$. Hepworth argues that this result is applicable to the following sequence of groups: \begin{enumerate}
%	\item The symmetric group $S_n$
%	\item The general linear group $\GL_n(\bZ)$.
%	\item The automorphism group $\Aut(F_n)$ of free groups.
%	\item The braid group $\mathrm{Br}_n$.
%\end{enumerate} We refer the readers to \cite{HepEdge} for more details and references. We believe that such a result should hold for any Koszul monoid and not just the one built out of trivial representations. We prove the following result in this direction and we show that \autoref{intro-APS} is a special case of it.

Our new proof of \autoref{intro-APS} follows from a more general result on Koszul monoids. 

\begin{theorem}
	\label{intro:general-vanishing}
	Let $A$ be a (skew) commutative Koszul monoid in $\prod_{n \ge 0} \Mod_{\bZ[\GL_n(R)]}$ where $R$ is a PID. Assume that  the following holds: \begin{enumerate}
		\item $\rH_0(\GL_2(R); A_2) = 0$.
		\item The product map $\rH_0(\GL_1(R); A_1) \otimes \rH_1(\GL_2(R); A_2) \to \rH_1(\GL_3(R); A_3)$ is surjective.
	\end{enumerate} Then we have the following: \begin{enumerate}[\indent (a')]
		\item The product map $\rH_0(\GL_1(R); A_1) \otimes \rH_i(\GL_{n-1}(R); A_{n-1}) \to \rH_i(\GL_n(R); A_n)$ is surjective for $n \ge 2i+1$.
		\item $\rH_i(\GL_n(R); A_n) = 0$ in degrees $n \ge 2i+2$.
	\end{enumerate}
\end{theorem}

%\begin{remark}

\autoref{intro:general-vanishing} can be generalized by replacing the groupoid $\{\GL_n(R) \}_{n \geq 0}$ with other braided monoidal groupoids such as the groupoids formed by braid groups or symmetric groups. In this generality, one can see that Hypotheses (a) and (b) are necessary by considering symmetric groups and the sign representation. See \autoref{signRep}.

%	The theorem above is not specific to $\GL_n(R)$ and the same argument carries over to many other reasonable sequences of groups but we chose not to provide a more abstract setting. This is in line with the theorem of Hepworth mentioned above. 
%\end{remark}

It is not known whether $\oSt$ is Koszul in $\prod_{n \ge 0} \Mod_{\bZ[\GL_n(\bZ)]}$. In particular,  the analogue of \autoref{intro-APS} for $\GL_{n}(\bZ)$ is not known.  However, Church--Putman proved the following homological vanishing theorem for $i=1$.

\begin{theorem}[Church--Putman, Theorem~A \cite{CP}]
	\label{thm:CP-vanishing-char-0}
	For $\bk$ a field of characteristic zero, we have: \begin{align*}
\rH_1(\GL_n(\bZ); \oSt_n(\bZ) \otimes_{\bZ}\bk) &= 0  \mbox{ for } n\ge 0, \\
 \rH_1(\SL_n(\bZ); \oSt_n(\bZ) \otimes_{\bZ}\bk) &= 0  \mbox{ for } n\ge 3,
	\end{align*} 
\end{theorem}

We use our main result, \autoref{thm:main}, to prove the following integral refinement. 

\begin{theorem} 
	\label{thm:integral-refinement}
	Let $\bk$ be an arbitrary commutative ring. For $n \ge 6$, we have that: \begin{align*}
	\rH_1(\GL_n(\bZ); \oSt_n(\bZ) \otimes_{\bZ}\bk) &= 0, \\
	\rH_1(\SL_n(\bZ); \oSt_n(\bZ) \otimes_{\bZ}\bk) &= 0.
	\end{align*} 
\end{theorem}

\begin{remark}
	Although $\GL_n(\bZ)$ is only a duality group with $\bQ$ coefficients and, for $n$ even, the Steinberg module is {\bf not} the dualizing module (see Putman--Studenmund \cite{PStu} for a description of the dualizing module), the groups $\rH_i(\GL_n(\bZ);\oSt_n(\bZ))$ are still of interest. For $R$ a ring of integers in a number field, there is a spectral sequence due to Quillen \cite{Q} with \[E^1_{ab}=\rH_b(\GL_a(R);\oSt_a(R)) \implies \rH_{a+b}(\Omega^{\infty-1}(K(R)) ). \] Here $\Omega^{\infty-1}(K(R))$ denotes the one-fold delooping of the infinite loop space associated to the algebraic $\rK$-theory spectrum of $R$. Many of the calculations of algebraic $\rK$-groups using this spectral sequence have relied on homological vanishing results to simplify the $E^1$-page (see e.g. \cite{LSK3,LSK45,SGGHSY, DSEVKM}). We hope this new vanishing result will also be useful for computations. 
	
	%\jeremy{Show this to Sander to make sure he does not object or sees errors.}
	
%	\rohit{I showed it to Sander and he said he has more references that we should add here. So we should send him an email.}
	
\end{remark} 

%\jeremy{In an earlier version, there was a paragraph about short exact sequences and how if two out of the three have homoloigcal stability, the third has rep stability and how this is analogous to what happens here. Rohit voted that this should not be in the paper. Peter, what are your thoughts?}

\subsection*{Outline of the paper}

In \autoref{sec:preliminaries}, we discuss algebraic  preliminaries and define the Steinberg monoid and the apartment monoid. In \autoref{intro-koszul}, we prove our Koszulness result for the Steinberg  monoid of a field (\autoref{intro-koszul-main}). In \autoref{sec:sharbly}, we use  Koszulness of the Steinberg monoid to establish finiteness properties of the Steinberg monoid viewed as a module over the apartment monoid. In \autoref{sec:congruence}, we use these finiteness properties to prove our main representation stability result for the codimension-one cohomology of level $3$ congruence subgroups (\autoref{thm:main}). We pose a refined version of a conjecture due to Church--Farb--Putman in \autoref{sec:stability-conjectures} that generalizes \autoref{thm:main} to all primes, all codimensions, and also addresses a more robust notion of representation stability. Finally, in \autoref{sec:homological-vanishing}, we give a new proof of the homological vanishing theorem of Ash--Putman--Sam (\autoref{intro-APS}), and an integral refinement of Church--Putman's homological vanishing theorem (\autoref{thm:integral-refinement}).

%In \autoref{sec:spectral-sequence}, we construct a spectral sequence which is similar in spirit to the ones employed by Putman--Sam \cite{PS} and Djament \cite{Dj}

\subsection*{Acknowledgments} We thank 
Avner Ash, Martin Bendersky, Thomas Church, Benson Farb, Paul Gunnells, Richard Hepworth, Alexander Kupers, Mark McConnell, Andrew Putman, Steven Sam, David Sprehn, and Jennifer Wilson for helpful conversations.  

\section{The apartment and Steinberg monoids}
\label{sec:preliminaries}

In this section, we define the apartment monoid and the Steinberg monoid. We begin by discussing the category of $\VB$-modules and its symmetric monoidal product. 
 
\subsection{$\VB$-modules}
Fix a commutative base ring $R$ and a commutative coefficient ring $\bk$.  Let $\VB$ (or $\VB_R$ when we need to be more precise) be the groupoid whose objects are finite-rank free $R$-modules and whose morphisms are $R$-linear bijections. There is a natural symmetric monoidal structure $\oplus$ on $\VB$ given by direct sum. A {\bf $\VB$-module} is a covariant functor $\VB \to \Mod_{\bk}$.  The symmetric monoidal structure $\oplus$ on $\VB$ induces a symmetric monoidal structure $\otimes$ on $\Mod_{\VB}$ given by:
\[ (M \otimes N)(X) = \bigoplus_{X_1 \oplus X_2 = X} M(X_1) \otimes_{\bk} N(X_2). \] To be more precise, in the expression above $X_1, X_2$ vary over free submodules of $X$ such that every element of $X$ can be written uniquely as a sum of a vector in $X_1$ and a vector in $X_2$. For $R$ a field, this reduces to the condition that $X_1 + X_2 = X$ and $X_1 \cap X_2 = 0$. See \autoref{rem:djament's-tensor} for further discussion. We reserve the symbol $\otimes$ without subscript for this product. Note that if $X_1$ and $X_2$ are rank $n$ free $R$-modules then there is a $\VB$-isomorphism $X_1 \to X_2$. We define $M(n)$ to be $M(R^n)$. We will often denote $\GL(R^n) = \GL_n(R)$ by $\GL_n$ when the ring is understood. There is an action of $\GL_n$ on $M(n)$ and a natural isomorphism:

\[ (M \otimes N)(n) \cong \bigoplus_{i + j = n} \Ind_{\GL_{i} \times \GL_{j}}^{\GL_{n}} M(i) \otimes_{\bk} N(j). \]   The support of a  $\VB$-module $M$ is the set of non-negative integers such that $M(n)$ is not zero. We say that $M$ is supported on a set if the support is contained in that set. If the support of $M$ is non-empty, we define the {\bf degree} of $M$ to be the supremum  of the support and otherwise define it to be $-1$. We denote the degree of $M$ by $\deg M$. 

The category $\Mod_{\VB}$ of $\VB$-modules is easily seen to be  equivalent to the product category $\prod_{n \ge 0}\Mod_{\bk[\GL(R^n)]}$, and so $\Mod_{\VB}$ is a Grothendieck category. In particular, it is complete and cocomplete.

\begin{remark}
	\label{rem:djament's-tensor}
	We note that the monoidal structure $\otimes$ on $\Mod_{\VB}$ is given by the composite: \[ \Mod_{\VB} \times \Mod_{\VB} \xrightarrow{\boxtimes} \Mod_{\VB \times \VB} \xrightarrow{\mathrm{Lan}_{\oplus}} \Mod_{\VB}   \] where $\boxtimes$ is the external tensor product and $\mathrm{Lan}_{\oplus}$ is the left Kan extension along $\oplus \colon \VB \times \VB \to \VB$. In particular, $\otimes$ is the tensor product $\otimes_{[\VB, \oplus]}$ defined in \cite{Dj}. This remarks allows us to quote results from \cite{Dj}.
\end{remark}

\subsection{The apartment monoid} \label{secAp}
Let $A$ be a monoid object in the category of  $\VB$-modules.   An $A$-module $M$  is a $\VB$-module together with an action $A \otimes M \to M$ satisfying the usual associativity and unitality conditions. If $M$ is a right $A$-module and $N$ is a left $A$-module, we define $M \otimes_{A} N$ as the coequalizer of the two natural maps: \[M \otimes A \otimes N \rightrightarrows M \otimes N.   \] Let $\Tor^{A}_*(M,N)$ denote the associated left-derived functors.  We say that  $A$ is skew commutative if $a \wedge b= (-1)^{nm}  (b\wedge a)$ whenever $a \in  A(R^n), b \in A(R^m)$, where $\wedge \colon A \otimes A \to A$ denote the multiplication map. We regard $\bk$ as the monoid in $\Mod_{\VB}$ concentrated in degree 0. All the monoids that we consider in this paper admit a natural surjective map of monoids  $A \to \bk$ whose kernel $A_+$ is supported in degrees $>0$ (see \autoref{subsec:koszul-def} for more on this). In particular, $\bk$ will have the structure of an $A$-module and so we can make sense of $\Tor^A_{\ast}(\bk, -)$.

Let $\lw$ denote the left adjoint of the forgetful functor from the category of skew commutative $\VB$-monoids to the category of $\VB$-modules. Let $\triv_1$ be the $\VB$-module such that $\triv_1(n) = 0$ for $n \neq 1$, and $\triv_1(1)$ is the rank $1$ trivial representation of $\GL_1$.  We call the skew commutative $\VB$-monoid $\lw(\triv_1)$ the {\bf apartment monoid}, and we denote it by $\bA$. See \autoref{remarkapartment} for a discussion of the name apartment monoid and the connection with apartments in the Tits building.  Unraveling the definition, we obtain the following result: 

\begin{proposition}
	\label{prop:apartment-definition}
	
Let $X$ be a free $R$-module of rank $n$. Then $\bA(X)$ is a $\bk$-module generated by elements of the form $[v_1, \ldots, v_n]$, one for each ordered basis $v_1, \ldots, v_n$ of $X$, subject to the following relations: \begin{enumerate}
		\item $[v_1, \ldots, v_n] =\sgn(\sigma)[v_{\sigma(1)}, \ldots, v_{\sigma(n)}] $ for $\sigma$ a permutation of $\{1,2,\ldots, n\}$.
		\item $[r v_1, v_2, \ldots, v_n] = [v_1, \ldots, v_n]$ for $r \in R^{\times}$.
	\end{enumerate} Moreover, the multiplication map  \[\bA(X_1) \otimes_{\bk} \bA(X_2)  \to \bA(X) \] is given on generators by \[ [u_1, \ldots, u_k] [v_1, \ldots, v_{n-k}] = [u_1, \ldots, u_k, v_1, \ldots, v_{n-k}]. \] 
\end{proposition}

%We now prove a lemma which will be useful for us later. This lemma is inspired by the proof of \cite[Theorem~A]{CP}.
%
%\begin{lemma}
%	\label{lem:CP-coinvariants-vanish}
%Assume that $2$ is invertible in $\bk$.	Let $M$ be an $\bA$-module generated in degrees $\le d$. Then $M(X)_{\SL(X)} = 0$ if the rank of $X$ is at least $d + 2$.
%\end{lemma}
%\begin{proof}
%	Since $M$ is generated in degrees $\le d$, there exists a $\VB$-module $V$ supported only in degree $d$ and a map $\bA \otimes V \to M$ which is surjective in degrees $\ge d$. Thus it suffices to prove the result for $M = \bA \otimes V$. For that, let $X$ be a free $R$-module of rank $n \ge d+2$. Then $M(X)$ is spanned by elements of the form $[v_1, \ldots, v_{n-d}] \otimes \alpha$ where $\{v_1, \ldots, v_{n-d}\}$ span an $R$-free summand $X_1$ of $X$, $\alpha \in V(X_2)$, and $X_1 \oplus X_2 = X$. Pick a basis $\{w_1, \ldots, w_d\}$ of $X_2$, and set $v_{n-d+i} = w_i$. Let $\sigma \in \SL(X)$ be the element defined on the basis $\{v_1, \ldots, v_n \}$ by \[\sigma(v_i) =\begin{cases}
%	 - v_2 &\mbox{if } i=1\\
%	 v_1  &\mbox{if } i=2 \\
%	 v_i  &\mbox{otherwise.}
%	\end{cases}   \] It is now clear that $\sigma([v_1, \ldots, v_{n-d}] \otimes \alpha) = -[v_1, \ldots, v_{n-d}] \otimes \alpha$. Since 2 is invertible in $\bk$, we see that the image of $[v_1, \ldots, v_{n-d}] \otimes \alpha$ in $\SL_n(X)$-coinvariants vanishes. This completes the proof.
%\end{proof}

In addition to exterior algebras, we will also need tensor algebras and symmetric algebras. The free (commutative) monoid functor is defined to be the left adjoint of the forgetful functor from the category of (commutative) monoid objects in $\Mod_{\VB}$ to $\Mod_{\VB}$. The tensor algebra $\rT(M)$ on a $\VB$-module $M$ is the free monoid on $M$. Similarly the symmetric algebra $\Sym M$ is the free commutative monoid on $M$. 

\begin{remark}
Monoid objects in a related representation theoretic monoidal category have been studied extensively by Sam--Snowden and others under the name twisted commutative algebras  (tcas); see \cite{SS1} for example. It is easy to spot the influence of the theory of tcas in this paper. 
\end{remark}

\subsection{The Steinberg monoid} 

Let $K$ be a field and $X$ an $n$-dimensional $K$-vector space. The Tits building for $X$ is denoted $\cT_n(X)$ and is the geometric realization of the poset of nonempty proper subspaces of $X$ ordered by inclusion. Given an integral domain $R$ and a free $R$-module $X$ of rank $n>0$, the Steinberg module with coefficients in $\bk$ is defined to be \[\St(X) := \widetilde \rH_{n-2}(\cT(X \otimes_R \Frac(R));\bk)\] where $\Frac(R)$ denotes the field of fraction. We define $\St(X)$ to be $\bk$ if the rank of $X$ is zero and use the usual convention that the reduced homology of the empty set is the coefficient ring. Note that here we are denoting the Steinberg module by a bold symbol which is in contrast to our notation in the introduction. We do this to take into account that now the coefficients are in a general ring $\bk$. We have chosen to introduce generalized coefficients to be able to talk about dimensions and characteristic zero settings.

Since $\GL(X \otimes_R \Frac(R)) \cong \GL_n( \Frac(R))$ acts on $\cT(X \otimes_R \Frac(R))$, $\St(X)$ is a representation of $\GL(X \otimes_R \Frac(R))$ and hence also a representation of $\GL(X) \cong \GL_n(R)$. For $X=R^n$, we denote $\St(X)$ by $\St_n(R)$. We now recall some presentations of Steinberg modules. 

\begin{theorem}[Lee--Szczarba {\cite[\S 3]{LS}}] \label{presentationsOfSteinberg}  Let $K$ be a field and $X$ an $n$-dimensional $K$-vector space. As a $\bk$-module, $\St(X)$ is generated by elements of the form $[v_1, \ldots, v_n]$, one for each ordered basis $v_1, \ldots, v_n$ of $X$, subject to the following relations: \begin{enumerate}
	\item $[v_1, \ldots, v_n] =\sgn(\sigma)[v_{\sigma(1)}, \ldots, v_{\sigma(n)}] $ for $\sigma$ a permutation.
	\item $[r v_1, v_2, \ldots, v_n] = [v_1, \ldots, v_n]$ for $r \in K^{\times}$.
	\item $\sum_i (-1)^i [v_0, v_1, v_2, \ldots, \hat v_i, \ldots, v_n] = 0$ where $v_0,\ldots, v_n$ are nonzero vectors and terms of the form $[v_0, v_1, v_2, \ldots, \hat v_i, \ldots, v_n] $ with $v_0, v_1, v_2, \ldots, \hat v_i, \ldots, v_n$ not a basis are omitted from the sum. 
\end{enumerate}

\noindent The $\GL(V)$ action is given by the formula $ g [v_1, \ldots, v_n]=[g v_1, \ldots, g v_n]$.  

\end{theorem}

Note that relation $(c)$ implies the first two relations. 

\begin{theorem}[Bykovskii \cite{Byk}] \label{presentationZ} Let $X$ be a free $\Z$-module of rank $n$. As a $\bk$-module, $\St(X)$ is generated by elements of the form $[v_1, \ldots, v_n]$, one for each ordered basis $v_1, \ldots, v_n$ of $\Z^n$, subject to the following relations: \begin{enumerate}
	\item $[v_1, \ldots, v_n] =\sgn(\sigma)[v_{\sigma(1)}, \ldots, v_{\sigma(n)}] $ for $\sigma$ a permutation.
	\item $[-v_1, v_2, \ldots, v_n] = [v_1, \ldots, v_n]$.
	\item $[v_1, v_2, \ldots, v_n] - [v_0, v_2, \ldots, v_n] + [v_0, v_1, \ldots, v_n] = 0$ where $v_0= v_1 + v_2$.

\end{enumerate}
\noindent The $\GL(X)$ action is given by the formula $ g [v_1, \ldots, v_n]=[g v_1, \ldots, g v_n]$. 
\end{theorem}

The Steinberg modules assemble to form a $\VB$-module. We now define a monoid structure on $\St$ when $R$ is a field or the integers. We call this the {\bf Steinberg monoid}. 

\begin{proposition} Let $R$ be $\Z$ or a field and let $X$ and $Y$ be free $R$-modules of rank $n$ and $m$ respectively.
The map $\St(X) \otimes_{\bk} \St(Y)  \m \St(X \oplus Y)$ given by \[ [v_1, \ldots, v_n] \otimes [u_1, \ldots, u_m]  \mapsto  [v_1, \ldots, v_n, u_1, \ldots, u_m] \] is well-defined and gives $\St$ the structure of a monoid object in $(\Mod_{\VB},\otimes)$.

\end{proposition} 

\begin{proof}
The only thing that is not trivial is that the function is well-defined. This follows from the above presentations. 
\end{proof}

A monoid structure on $\St$ actually exists for all integral domains but we will not need this. It follows \autoref{presentationsOfSteinberg} and \autoref{presentationZ} that for $R$ a field or the integers, there is a natural surjective map $\bA \m \St$ which is a map of monoids (in fact such a surjection exists for $R$ Euclidean by the work of Ash--Rudolph \cite{AR}).  This gives $\St$ the structure of an $\bA$-module. In fact, \autoref{presentationZ} (or \autoref{fieldPresentation} in the field case) immediately implies that $\St$ is a quotient of $\bA$ by a two-sided ideal generated in degree 2 (the additional relation (c) only depends on $v_0, v_1,v_2$ which span a rank 2 summand). From this perspective the assertion of the proposition above is immediate as well.

\begin{remark}\label{remarkapartment}
For every integral domain $R$, there is a natural map $\bA \m \St$ which is not always surjective. The image of a generator $[v_1,\ldots, v_n]$ in $\widetilde \rH_{n-2}(\St_n(R))$ is known as an {\bf apartment class} and is the fundamental class of a sphere in $\cT(\Frac(R)^n)$ known as an {\bf apartment}. The apartment associated to $[v_1,\ldots, v_n]$ is the full subcomplex of $\cT(\Frac(R)^n)$ with vertices given by the span of nonempty proper subsets of $\{v_1,\ldots, v_n \}$. When $R$ is  the ring of integers in $\Frac(R)$, these classes are known as {\bf integral apartment classes}. It is an interesting question to classify when the Steinberg module is generated by integral apartment classes. See \cite{AR,CFPint, MPWY} for more on this question. 

% is the apartment 

%Part of the Solomon--Tits theorem is the result that the Steinberg module of a field is generated by so-called apartment classes. Apartments in the Tits building are spheres which correspond to  direct sum decompositions of the vector space into lines. The monoid $\bA$ is generated as an abelian group by symbols $[v_1, \ldots, v_n]$. The images of these generators under the map $\bA_n \m \St_n$ are the apartment classes. This is the reason we call $\bA$ the apartment monoid. 
%\jeremy{If we picked the other definition of the Steinberg monoid, then A would always be generated by apartments but St would not be and the name apartment monoid would make more sense.}
\end{remark}

The following theorem seems to be known to experts but we could not find a reference for it in the literature so we will sketch a proof.

\begin{theorem} \label{fieldPresentation} Let $K$ be a field and $X$ an $n$-dimensional $K$-vector space. As a $\bk$-module, $\St(X)$ is generated by elements of the form $[v_1, \ldots, v_n]$, one for each basis $v_1, \ldots, v_n$ of $X$, subject to the following relations: \begin{enumerate}

	\item $[v_1, \ldots, v_n] =\sgn(\sigma)[v_{\sigma(1)}, \ldots, v_{\sigma(n)}] $ for $\sigma$ a permutation.
	\item $[rv_1, v_2, \ldots, v_n] = [v_1, \ldots, v_n]$ for $r\in K^\times$.
	\item $[v_1, v_2, \ldots, v_n] - [v_0, v_2, \ldots, v_n] + [v_0, v_1, \ldots, v_n] = 0$ where $v_0= v_1 +  v_2$.

\end{enumerate}
\noindent The $\GL(X)$ action is given by the formula $ g [v_1, \ldots, v_n]=[g v_1, \ldots, g v_n]$.  
\end{theorem}

\begin{proof}

Call $S(X)$ the $\bk[\GL(X)]$-module with the above presentation. By \autoref{presentationsOfSteinberg}, sending a generator to the generator with the same name gives a well-defined surjective map $f\colon S(X) \m \St(X)$. Pick an isomorphism $X \cong K^n$. To see that $f$ is injective, note that we can use relations (a), (b), and (c) to show that $S(K^n)$ is generated as a $\bk$-module by symbols $[v_1, v_2, \ldots, v_n]$ where the vectors assemble to form a unit upper triangular matrix with respect to the standard basis of $K^n$. The Solomon--Tits theorem (see \cite[Theorem~IV.5.2]{Brown-buildings}) implies that  $\St_n(K)$ has a basis given by unit  upper triangular matrices. Thus, $f$ is also injective.
\end{proof}

\section{Koszulness of the Steinberg monoid of a field}
\label{intro-koszul}

In this section, we prove that the Steinberg monoid of a field is Koszul. We follow an argument of Priddy \cite{priddy-koszul}.

\subsection{Preliminaries on Koszulness}
\label{subsec:koszul-def}
In a non-negatively graded monoidal category with unit object $\bk$ (which we assume is supported in degree 0), we say that a monoid $A$ is an {\bf augmented monoid} if there is a surjection $A \to \bk$ of monoids (called the {\bf augmentation map})  whose kernel $A_+$ (called the {\bf augmentation ideal}) is supported in degrees $>0$. Note that if such an augmentation exists, it will be unique. A map of augmented monoids is a map of monoids that preserves the augmentation map. In particular, the augmentation map defined above is a map of augmented monoids. The apartment and Steinberg monoids are naturally augmented monoids in $(\Mod_{\VB}, \otimes)$. One can define Koszul monoids in the general setting of symmetric monoidal categories but we discuss it only in the setting of $\VB$-modules for concreteness as follows.

Recall that we regard $\bk$ as a $\VB$-module supported in degree $0$. In other words, $\bk$ is the unit object in $\Mod_{\VB}$, and so is naturally an augmented monoid. Let $A$ be an augmented monoid in $(\Mod_{\VB}, \otimes)$. The two-sided reduced bar resolution $\cB_{\ast}(A, A) \to A \to 0$ of $A$ is given by \[ A \otimes \rT^{\ast}(A_{+}) \otimes A \to A \to 0  \] where $\rT(.)$ denotes the tensor algebra and $\rT^{\ast}(.)$ denotes its $*$th graded piece (we have borrowed our notation from \cite{priddy-koszul}). If $M$ is a right $A$-module and $N$ is a left $A$-module, then $\Tor_{\ast}^A(M, N )$ is the homology of \[\cB_{\ast}(M, A, N) \coloneq M \otimes_{A} \cB_{\ast}(A, A) \otimes_{A} N.\] The module  $ \cB_{s}(M, A, N)$ is generated by elements of the form $m \otimes a_1 \otimes\cdots \otimes a_s \otimes n$ where $m \in M$, $n \in N$ and $a_i \in A_{+}$.  Such elements are written as $m \otimes [a_1 | a_2| \cdots | a_s] \otimes n$ for historical reasons. We shall only need the special case when $M =N =\bk$. We denote the complex  $ \cB_{\ast}(\bk, A, \bk)$ calculating $\Tor_{\ast}^A(\bk, \bk)$ simply by $\ol{\cB}_{\ast}(A)$. In this case, the differential is given by \[\partial([a_1 | a_2| \cdots | a_s]) = \sum_{j=1}^{s-1} (-1)^{j-1} [a_1 | \cdots | a_j a_{j+1} | \cdots | a_s].  \]  For every element $[a_1 | a_2| \cdots | a_s] \in \ol{\cB}_{\ast}(A)$, the homological degree is defined to be $s$ and internal degree is defined to be $\sum_{i=1}^s \deg a_i$.  We shall denote the homological degree $s$ and internal degree $n$ piece of $\ol{\cB}_{\ast}(A)$ by $\ol{\cB}_s^n(A)$. We say that $A$ is  {\bf Koszul} if $\Tor_i^{A}(\bk, \bk)$ is supported only in internal degree $i$ (for each $i \ge 0$). Equivalently, $A$ is a Koszul monoid if, for each $n \ge 0$, the homology of $\ol{\cB}^n_{\ast}(A)$ is supported in homological degree $n$. 

\subsection{Koszulness of the apartment and the Steinberg monoids} Recall that the apartment monoid $\bA$ is by definition $\lw(\triv_1)$, the exterior algebra on $\triv_1$. Koszulness of exterior algebras is quite well-known in other contexts and the usual proof carries over to the symmetric monoidal category of $\VB$-modules.
\begin{theorem}[Koszulness of the apartment monoid] 
\label{thm:apartment-monoid-is-Koszul}	
Let $R$ be a PID, and $\bA$ be the apartment monoid in $(\Mod_{\VB_R}, \otimes)$. Then $\bA$ is Koszul. 
\end{theorem}

%We shall also use the following Koszulness result.

%\begin{theorem}[Koszulness of the tensor algebra] 
%	\label{thm:tensor-algebra-is-Koszul}Let $E'$ be the $\VB$-module such that $E'(n) = 0$ for $n \neq 1$, and $E'(1)$ is the regular representation of $\GL_1$. If $R$ is a PID then the tensor algebra $\rT(E')$ is Koszul. 
%\end{theorem}

We shall not prove \autoref{thm:apartment-monoid-is-Koszul}
% or \autoref{thm:tensor-algebra-is-Koszul} 
but an interested reader can obtain proofs by following the proof of Koszulness of the Steinberg monoid $\St$. Also see the discussion in \autoref{sec:koszul-discussion}. We prove Koszulness of $\St$ in the case when $R$ is a field as follows: \begin{itemize}
	\item Using the Solomon--Tits theorem, we obtain a basis of the Steinberg module consisting of unit upper triangular matrices. 
	
	%An application of \autoref{thm:church--putman--By} yields a basis for the Steinberg monoid and $\ol{\cB}_{\ast}(\St)$
	% \jeremy{I am cutting most thm:church--putman--By. I think every time we quote it, it is just to say our Steinberg equals other people's Steinberg. I think we no longer need to do this since it is true by definition. However, be on the look out in case I messed stuff up by cutting references to this theorem.}
	\item We define a well-ordering on the basis of $\ol{\cB}_{\ast}(\St)$ obtained above which makes $\St$ a ``Poincar\'e--Birkhoff--Witt-like'' (PBW-like) monoid (we chose not to make this precise). 
	\item We follow Priddy's argument as in \cite{priddy-koszul} that PBW algebras are Koszul to finish our argument. 
\end{itemize}

\begin{theorem}[Koszulness of the Steinberg monoid] 
	\label{thm:Koszulness-of-Steinberg} Suppose $R$ is a field. Then $\Tor_i^{\St}(\bk, \bk)$ is supported in degree $i$.
\end{theorem}

Fix a field $K$, and assume that $R = K$ throughout the rest of this subsection.  \autoref{prop:homotopy} below immediately implies the theorem above. We now provide some preliminaries needed for the proposition.

\begin{proposition}
	\label{prop:unipotent-basis-of-steinbergs}
	Let $W$ be a $d$-dimensional $K$-vector space. Let $B = (v_1, \dots, v_d)$ be an ordered basis of $W$, and let $\bU_d$ be the  subgroup of those matrices in $\GL_d(K)$ whose matrix with respect to $B$ is unit lower triangular. So if $g \in \bU_d$ then  $ g v_i =  v_i + \sum_{j > i} c_{ij} v_j$ for some $c_{ij} \in K$. 	 Then $\St(W)$ is freely generated as a $\bk$-module by the apartment classes \[ \PBW_W \coloneq   \{  [g v_1, \ldots, g v_d] \mid g \in \bU_d  \}.\]
\end{proposition}
\begin{proof} This is the Solomon--Tits theorem (see \cite[Theorem~IV.5.2]{Brown-buildings}).
\end{proof}

\begin{remark}
	\label{rem:lower-triangular}
	Let $W$ be a $d$-dimensional subspace of $K^n$. Let $B = (v_1, \dots, v_d)$ be an ordered basis of $W$, and let $M$ be a matrix whose $i$-th column is $v_i$ written in the standard basis of $K^n$. Then $\PBW_{W}$ consists of column vectors in matrices obtained by multiplying $M$ with a $d \times d$ unit lower triangular matrix on the right.
\end{remark}

To prove \autoref{thm:Koszulness-of-Steinberg}, it suffices to show that the homology of $\ol{\cB}^n_{\ast}(\St)$ is supported in homological degree $n$. We now describe $\ol{\cB}^n_{\ast}(\St)$ in more detail. Since we are only interested in internal degree $n$, we can work with the  $K$-vector space $K^n$ and its subspaces. In particular, we note that \[\ol{\cB}^n_{s}(\St) = \bigoplus_{\bigoplus_{j = 1}^s W_j = K^n} \St(W_1) \otimes_{\bk} \cdots \otimes_{\bk} \St(W_s), \] where we note that a simple tensor $a_{W_1} \otimes \cdots \otimes a_{W_s}$ in $\St(W_1) \otimes_{\bk} \cdots \otimes_{\bk} \St(W_s)$ is denoted as $[a_{W_1}| \cdots | a_{W_s}]$ as is conventional for the Bar construction.  Given a subspace $W$ of $K^n$, we describe below a PBW-like basis of $\St(W)$. Let the dimension of $W$ be $d \le n$. We now assign to $W$ a canonical subset $S_W \subset [n]$ of size $d$, a canonical  $K$-basis $B_W$ of $W$, and a canonical PBW-like $\bk$-basis, denoted $\PBW_W$, of $\St(W)$: \begin{enumerate}
	\item Let $v_1, \ldots, v_d$ be a basis of $W$. Put $v_1, \ldots, v_d$ as column vectors (with respect to the standard basis of $K^n$)  in an $n \times d$ matrix $M$. 
	\item Let $N$ be the column-reduced Echelon form of $M$. Recall that the column-reduced Echelon form (transpose of a row-reduced Echelon form) is unique and does not depend on the choice of $v_1, \ldots, v_d$ made earlier. 
	\item Define $S_W$ to be the set of row indices that contain a leading one in $N$, and define $B_W$ to be the columns of $N$. Note that there is a natural ordering on $B_W$ coming from column indices.
	\item  Define $\PBW_{W}$ to be the basis as described in \autoref{prop:unipotent-basis-of-steinbergs} with respect to the ordered basis $B_W$. 
\end{enumerate}

Note here that we can think of $S_W$ as the index of the lexicographically least nonzero Pl\"ucker-coordinate of $W$. As an example, suppose $n = 4$ and let $W$ be the subspace of $K^4$ of dimension $d= 3$ generated by the columns of the following matrix $M$:  \[M = \begin{bmatrix}
	1 & 0 & 0\\
	2 & 0 & 0\\
	0 & 1 & 0\\
	0 & 0 & 1
\end{bmatrix}\]  Since $M$ is in column-reduced Echelon form already, we have $N=M$. By definition, we have $S_W = \{1, 3,4 \}$, and $B_W$ consists of columns of $M$.  Moreover, $\PBW_{W}$ consists of columns matrices of the following form (see \autoref{rem:lower-triangular}): \[ 
 \begin{bmatrix}
	1 & 0 & 0\\
	2 & 0 & 0\\
	0 & 1 & 0\\
	0 & 0 & 1
\end{bmatrix}
\begin{bmatrix}
1 & 0 & 0  \\
\ast & 1 & 0  \\
\ast & \ast & 1 
\end{bmatrix}\] 

We have a partial order $\prec$ on subsets of $[n]$ given by $S_1 \prec S_2$ if $\max S_1 < \min S_2$. We now list some crucial properties of our construction above: \begin{itemize}
	\item[(P1)] If $W_1 \oplus W_2$ is a summand of $K^n$ and $S_{W_1} \prec S_{W_2}$, then $S_{W_1  \oplus W_2} = S_{W_1} \sqcup S_{W_2}$.
	\item[(P2)] If $W_1 \oplus W_2$ is a summand of $K^n$ and $S_{W_1} \prec S_{W_2}$, then the multiplication \[\St(W_1) \otimes \St(W_2) \to \St(W_1 \oplus W_2)\] takes $\PBW_{W_1} \times \PBW_{W_2}$ inside $\PBW_{W_1 \oplus W_2}$.
	\item[(P3)] Given a partition $S_1 \sqcup S_2$ of $S_W$ with $S_1 \prec S_2$, there is a natural map \[\PBW_{W} \to \coprod_{W_1\oplus W_2 = W}\PBW_{W_1} \times \PBW_{W_2}\] which we now describe. Suppose the ordered basis is $B_W = (v_1, \ldots, v_d)$. Then every element in $\PBW_{W}$ can be written uniquely in the form $[g v_1, g v_2, \ldots, g v_d]$ for some $g \in \bU_d$. Suppose the sizes of $S_1$ and $S_2$ are $d_1$ and $d_2$ respectively, then the claimed natural map is given by \[ [g v_1, g v_2, \ldots, g v_d] \mapsto ([g v_1 , g v_2, \ldots, g v_{d_1}], [g v_{d_1 + 1}, \ldots, g v_d]).  \] The corresponding $W_1$ and $W_2$ are given by the spans of $\{g v_1, gv_2, \ldots, gv_{d_1}\}$ and $ \{ gv_{d_1 + 1}, \ldots, gv_d \}$, respectively. Note that then $S_{W_1} = S_1$ and $S_{W_2} = S_2$.
\end{itemize}

\begin{remark}
	We restrict to the case when $R = K$ is a field because the property (P3) as above does not have an analogue for more general rings. For example, the issue when $R = \bZ$  is the existence of a finite-rank free abelian group $W$ together with a decomposition $A \oplus B$ of $W \otimes_{\bZ}\bQ$ such that $(W \cap A)  \oplus (W \cap B) \neq W$. 
\end{remark}

We now follow Brunetti--Ciampella's argument \cite[Theorem~2.5]{koszul-not-cobar} which in turn is based on an argument of Priddy \cite[Theorem~5.3]{priddy-koszul} to complete our proof of Koszulness  with \autoref{prop:homotopy}.  

%\begin{definition}
%	We define the {\emph irreducibility index} of an element \[ x = [a_{W_1} | \cdots | a_{W_s}] \in \ol{\cB}^n_{s}(\St)  \] to be the integer \begin{align*}
%	\ai(x) = \begin{cases}
%	s &\mbox{if } C = \emptyset \\
%	\min C &\mbox{if } C \neq \emptyset
%	\end{cases}
%	\end{align*} where $C = \{j \colon  S_{W_j} \cap S_{W_{j+1}} = \emptyset \}$.
%\end{definition}

\begin{proposition}
	\label{prop:homotopy}
	 $\rH_s(\ol{\cB}^n_{\ast}(\St)) = 0$ for $s \neq n$.
\end{proposition}
\begin{proof}
Because $\ol{\cB}^n_{\ast}(\St)$ is supported in homological degrees $\le n$, it is enough to prove that $\rH_s(\ol{\cB}^n_{\ast}(\St)) = 0$ for $s < n$.

	We define a $\bk$-linear map $\Phi \colon \ol{\cB}^n_{s}(\St) \to \ol{\cB}^n_{s+1}(\St)$ on every $x = [a_{W_1} | \cdots | a_{W_s}]$ with $a_{W_j} \in \PBW_{W_j}$ for each $j$ (such elements form a $\bk$-basis of $\ol{\cB}^n_{s}(\St)$) as follows: \begin{enumerate}
		\item Set $k$ to be the smallest index such that $\rank W_k > 1$. (If such $k$ does not exist, $\Phi(x) = 0$ automatically because $x$ is in top degree.) We call this the \textbf{widening index} of $x$.
		\item If $S_{W_j} \prec S_{W_{j+1}}$, we call $j$ an \textbf{orderpreserving index} of $x$. If there is an orderpreserving index $j$ strictly smaller than $k$, set $\Phi(x) = 0$. {\bf Example}: Suppose $n=4$, and let $x = [a_{W_1} | a_{W_2} | a_{W_3}] \in \ol{\cB}^n_{3}(\St)$ be given by \[ x = [\begin{bmatrix}
		1\\
		0\\
		 0\\
		0
		\end{bmatrix} \vert \begin{bmatrix}
		0\\
		1\\
		0\\
		0
		\end{bmatrix} \vert \begin{bmatrix}
		0 & 0\\
		0 & 0\\
		1 &0 \\
		0 &1
		\end{bmatrix}].  \] Then $k =3$, and $S_{W_1} = \{1 \} \prec \{2\} = S_{W_2}$. So $j =1$ is an orderpreserving index strictly smaller than $k$.
		\item Otherwise, let $m = \min S_{W_{k}}$ and $S$ be the complement of $m$ in $S_{W_{k}}$. 
		\item  By property (P3), the partition $\{m\} \sqcup  S$ gives rise  to a map 
		\[ \PBW_{W_k} \to \coprod_{W\oplus W'=W_k} \PBW_{W} \times \PBW_{W'}.\]
		Let $(b,b')$ be the image of $a_{W_k}$ under this map. 
		\item Set $\Phi(x) = (-1)^{k-1} [a_{W_1}|\cdots| a_{W_{k-1}} | b | b' | a_{W_{k+1}}| \cdots | a_{W_s} ]$. {\bf Example}: Suppose $n=4$, and let $x = [a_{W_1} | a_{W_2} | a_{W_3}] \in \ol{\cB}^n_{3}(\St)$ be given by \[ x = [\begin{bmatrix}
		0\\
		0\\
		0\\
		1
		\end{bmatrix} \vert \begin{bmatrix}
		0\\
		0\\
		1\\
		0
		\end{bmatrix} \vert \begin{bmatrix}
		1 & 0\\
		6 & 1\\
		3 &0 \\
		5 &0
		\end{bmatrix}].  \] Then $k =3, S_{W_1} = \{4 \}, S_{W_2} = \{3 \}$, and $S_{W_3} = \{1,2\}$. So $m =1$, and $S =\{2 \}$. In this case $\Phi(x)$ is given by \[ \Phi(x) = (-1)^2[\begin{bmatrix}
		0\\
		0\\
		0\\
		1
		\end{bmatrix} \vert \begin{bmatrix}
		0\\
		0\\
		1\\
		0
		\end{bmatrix} \vert \begin{bmatrix}
		1 \\
		6 \\
		3\\
		5
		\end{bmatrix} \vert \begin{bmatrix}
		 0\\
		1\\
		0 \\
		0
		\end{bmatrix}].  \]
	\end{enumerate}
	
	Let us furthermore define a filtration on $\ol{\cB}^n_{s}(\St)$:
	\begin{enumerate}
		\item To every basis element $x = [a_{W_1} | \cdots | a_{W_s}]$ with $a_{W_j} \in \PBW_{W_j}$ we associate a word $w_x$ in $[n]^n$ by concatenating $S_{W_1}, \dots, S_{W_s}$, where the elements of $S_{W_j}$ are ordered as natural numbers.
		\item Let $w = (i_1, \ldots, i_n)$ be a sequence in $[n]^n$. We say that a pair $(\alpha, \beta)$ such that $\alpha <\beta$ is an inversion for $w$ if $i_{\alpha} \ge i_{\beta}$. Define a quasi order $<$ on $[n]^n$ by $w < w'$ if $w$ has more inversions than $w'$. 
	
	Note that there is a unique maximal element in $[n]^n$ given by $(1,2, \ldots, n)$ which corresponds, for example, to $x = [\be_1 | \be_2| \ldots| \be_n]$ where $\be_1, \ldots, \be_n$ is the coordinate basis of $K^n$ (or to any basis element $x$ which is lower triangular with the notation as in the example above). On the other hand, if $x = [\be_n | \be_n+\be_1| \be_n+\be_2| \ldots| \be_n + \be_{n-1}]$ then $w_x = (n, n, \ldots, n)$ which is a minimal element in $[n]^n$. Intuitively, $w_x > w_y$ if $x$ is closer to being a lower triangular matrix compared to $y$.
		
%		\item Let $[n]^n$ be ordered lexicographically, i.e. $w < w'$ if $w_i < w'_i$ for some $i \in [n]$ and $w_j = w'_j$ for all $j<i$. 

		\item Let $F_{\le \alpha}\ol{\cB}^n_{s}(\St)$ be generated by all basis elements $x$ such that the number of inversions in $x$ is at most $\alpha$.
	\end{enumerate}

For an element $y \in \ol{\cB}^n_{s}(\St)$ we say that $w_y > w$ if $y$ can be written as a sum of basis elements $x$ (as in (a) above) satisfying $w_x > w$. A crucial property of this order is the following. Let $x = [a_{W_1} | \cdots | a_{W_s}]$ be a basis element and let  $y = [a_{W_1} | \cdots | a_{W_{j-1}} |a_{W_j} a_{W_{j+1}} | a_{W_{j+2}} | \cdots | a_{W_s}]$, where $a_{W_j} a_{W_{j+1}}$ is the image of $a_{W_j} \otimes a_{W_{j+1}}$ under the map $\St(W_j)\otimes \St(W_{j+1}) \to \St(W_1\oplus W_2)$. Then we have $w_x \le  w_y$ and the equality holds if and only if $S_{W_j} \prec S_{W_{j+1}}$ (if and only if $y$ is a basis element). In particular, \[ 0 = F_{\le -1}\ol{\cB}^n_{\ast}(\St) \subset F_{\le 0}\ol{\cB}^n_{\ast}(\St) \subset F_{\le 1}\ol{\cB}^n_{\ast}(\St) \subset \ldots \subset F_{\le \binom{n}{2}}\ol{\cB}^n_{\ast}(\St) = \ol{\cB}^n_{\ast}(\St) \]  defines a finite increasing filtration on the chain complex $\ol{\cB}^n_{\ast}(\St)$. We also note here that $w_{\Phi(x)} = w_{x}$ for every basis element $x$ for which $\Phi(x)$ is nonzero.
	
	In the remainder of the proof, we will verify that $\partial \Phi + \Phi \partial  -\id$ sends $F_{\le \alpha}\ol{\cB}^n_{s}(\St)$ to $F_{<  \alpha}\ol{\cB}^n_{s}(\St)$ for all $s<n$. This then shows that 
	\[ \rH_s\Big( F_{\le \alpha}\ol{\cB}^n_{*}(\St)/ F_{< \alpha}\ol{\cB}^n_{\ast}(\St) \Big) = 0 \quad\text{for $s<n$,}\]
	because every cycle 
	\[c = (\partial \Phi + \Phi \partial )(c) = \partial (\Phi(c))\] 
	in degree less than $n$ is also a boundary. By induction on the filtration, we deduce that $\rH_s(\ol{\cB}^n_{*}(\St)) = 0$ for $s<n$.
	%%%
	We prove the claim by a complete case study. Let $x = [a_{W_1} | \cdots | a_{W_s}]$ with $a_{W_j} \in \PBW_{W_j}$ and $s<n$. Set $k$ to be the smallest index such that $\rank W_k > 1$.
	
	\textbf{Case 1:} Assume there is a $j<k$ such that $S_{W_j} \prec S_{W_{j+1}}$. Also let $j_0$ be the smallest orderpreserving index. Then $(\partial \Phi + \Phi \partial)(x) = \Phi \partial(x)$. Let $y_i = [ a_{W_1}| \dots| a_{W_i}a_{W_{i+1}} | \dots| a_{W_s}]$, so that
	\[\partial(x) = \sum_{i=1}^{s-1} (-1)^{i-1}y_i.\]
	First we see that $\Phi(y_{j_0}) = x$, because $j_0$ is the widening index of $y_{j_0}$ and there is no orderpreserving indices smaller than $j_0$. For $i<j_0$, we get that $w_{y_i} > w_x$, because $i$ is not an orderpreserving index of $x$. Hence $w_{\Phi(y_i)} = w_{y_i}>w_x$. If $i = j_0+1$, we have two cases. If $j_0$ is an orderpreserving index of $y_{j_0+1}$, then $\Phi(y_{j_0+1}) = 0$ because its widening index is $j_0+1$. It is possible that $j_0$ is not an orderpreserving index if $S_{W_{j_0}} \not \prec S_{W_{j_0+2}}$, but then $j_0+1$ cannot be orderpreserving and $w_{\Phi(y_{j_0+1})}=w_{y_{j_0+1}}>w_x$. Finally, if $i\ge j_0+2$, $j_0$ is orderpreserving and smaller than the widening index of $y_i$. Therefore $\Phi(y_i) = 0$. 
	
	\textbf{Case 2:} Assume there is no $j<k$ such that $S_{W_j} \prec S_{W_{j+1}}$. Let $y_i = [ a_{W_1}| \dots| a_{W_i}a_{W_{i+1}} | \dots| a_{W_s}]$. Let us write $\Phi(x) = (-1)^{k-1} [a_{W_1}|\cdots| a_{W_{k-1}} | a_{W_k'} | a_{W_k''} | a_{W_{k+1}}| \cdots | a_{W_s} ] $ as in its definition, that is, $(a_{W_k'}, a_{W_k''})$ is the image of $a_{W_k}$ under the map \[ \PBW_{W_k} \to \coprod_{W_k'\oplus W_k''=W_k} \PBW_{W_k'} \times \PBW_{W_k''}.\] We now define  
	\[ z_i = \begin{cases} [a_{W_1}|\dots| a_{W_i}a_{W_{i+1}} |\cdots| a_{W_{k-1}} |a_{W_k'} | a_{W_k''} | a_{W_{k+1}}| \cdots | a_{W_s} ] &\text{if $i<k-1$,}\\
	[a_{W_1}|\dots| a_{W_i}a_{W_{i+1}} |\cdots| a_{W_{k-1}}  a_{W_k'} | a_{W_k''} | a_{W_{k+1}}| \cdots | a_{W_s} ] &\text{if $i = k-1$,}\\
	[a_{W_1}|\cdots| a_{W_{k-1}} | a_{W_k'} a_{W_k''} | a_{W_{k+1}}| \cdots | a_{W_s} ] &\text{if $i= k$,}\\
	[a_{W_1}|\cdots| a_{W_{k-1}} | a_{W_k'} |a_{W_k''}  a_{W_{k+1}}| \cdots | a_{W_s} ] &\text{if $i= k+1$,}\\
	[a_{W_1}|\dots| a_{W_{k-1}} | a_{W_k'} | a_{W_k''} | a_{W_{k+1}}| \cdots | a_{W_{i-1}}a_{W_i} |\cdots| a_{W_s} ] &\text{if $i>k +1$.}
	\end{cases}\]
	Then
	\[\partial(x) = \sum_{i=1}^{s-1} (-1)^{i-1}y_i\quad\text{and}\quad\partial\Phi(x) = \sum_{i=1}^{s} (-1)^{i-1}(-1)^{k-1}z_i.\] Then we have the following: \begin{itemize}
	\item	For $i < k$, $w_{\Phi(y_i)}=w_{y_i}$ is larger than  $w_x$ as $S_{W_i} \not\prec S_{W_{i+1}}$ ($i$ is not an orderpreserving index of $x$).
	
	\item For $i < k-1$, $w_{z_i}$ is larger than $w_x$ again because $S_{W_i} \not\prec S_{W_{i+1}}$ ($i$ is not an orderpreserving index of $x$).
	
	\item $w_{z_{k-1}}$ is larger than $w_x$ as $S_{W_{k-1}} \not\prec S_{W_{k}'}$  (recall that $S_{W_{k}'} = \min S_{W_k}$ and $k-1$ is not an orderpreserving index for $x$).
	
	\item $z_k = x$ as $S_{W_k'} \prec S_{W_k''}$ by construction (in other words, $a_{W_k'} a_{W_k''} =a_{W_k}$ is already a basis element in $\St(W_k)$).
	
	\item  If $z_{k+1}$ and $\Phi(y_k)$ do not cancel, $k$ cannot be an orderpreserving index of $x$, which implies that both $w_{\Phi(y_k)}=w_{y_k}$ and $w_{z_{k+1}}$ are larger than $w_x$. 
	
	\item $\Phi(y_i)$ cancels with $z_{i+1}$ in $(\partial \Phi + \Phi \partial)(x)$ if $i> k$. To see this, just note that the widening index for each term in $y_i$ expressed in the basis (a) has widening index $k$ and orderpreserving index $\ge k$, and so we have \[\Phi(y_i) = (-1)^{k-1}[a_{W_1}|\dots| a_{W_{k-1}} | a_{W_k'} | a_{W_k''} | a_{W_{k+1}}| \cdots | a_{W_{i}}a_{W_{i+1}} |\cdots| a_{W_s} ] \] even if $a_{W_{i}}a_{W_{i+1}}$ is not a basis element.
	\end{itemize} This finishes the case study and the proof.
\end{proof}

The proposition above immediately implies \autoref{thm:Koszulness-of-Steinberg}.

%\jeremy{Maybe it would be good to have an example environment that takes 1/3 to 2/3 of a page and has lots of examples of all of the types of things that happen in this section (the filtrations, bases, partial orders, functions, etc.).}

\begin{question}
Can one prove a version of \autoref{thm:Koszulness-of-Steinberg} when $R$ is not a field, for example, when $R= \bZ$? A version of \autoref{thm:Koszulness-of-Steinberg} or even \autoref{thm:shably-main} for $R=\bZ$ would establish the Church--Farb--Putman conjecture \cite[Conjecture 2]{CFPconj}  on vanishing of the high-dimensional rational cohomology of $\SL_n(\Z)$ (except with a slightly worse range). 
\end{question}

\subsection{Koszul resolutions}
\label{sec:koszul-discussion}
Let $A$ be an augmented monoid in $\Mod_{\VB}$ such that $A(X)$ is $\bk$-flat for each $X$. Then the functor $- \otimes A$ is exact and is the left adjoint to the restriction functor $\Mod_A \to \Mod_{\VB}$.  This adjunction follows from the Yoneda lemma as an $A$-module is the same as a functor $\cC \to \Mod_{\bk}$ where $\cC$ is the category whose objects are finite-rank free $R$-modules and whose morphisms $X \to Y$ are triples $(f, C, a)$ where $f$ is an injection of $R$ modules such that  $Y = X \oplus C$, and $a \in A(C)$. Since the restriction functor is exact, we see that the $A$-module $V \otimes A$ is projective for any projective $\VB$-module $V$. It follows that the $A$-module $V \otimes A$ is $- \otimes_A \bk$-acyclic for any $\VB$-module $V$. If $A$ is Koszul, then $\Tor_i^A(\bk, \bk)$ is concentrated in degree $i$ (for each $i \ge 0$). So, by a dimension shifting argument,  one can construct a linear resolution \[ V_{\ast} \otimes A \to \bk \to 0 \] such that $V_i$ is a $\VB$-module concentrated only in degree $i$.  The converse is also true, that is, if such a resolution exists then $A$ is Koszul.  This resolution is called the {\bf Koszul resolution}. In fact, we can take $V_i = \Tor_i^{A}(\bk, \bk)$ in the Koszul resolution. As an example, let $\reg_1$ be the $\VB$-module such that $\reg_1(n) = 0$ for $n \neq 1$, and $\reg_1(1)$ is the regular representation of $\GL_{1}$. The following sequence is clearly exact:  \[ 0 \to  \rT(\reg_1)\otimes \reg_1 \to \rT(\reg_1) \to \bk \to 0.  \] In particular, $\Tor_1^{ \rT(\reg_1)}(\bk, \bk) = \reg_1$ and  $\Tor_i^{ \rT(E')}(\bk, \bk) = 0$ for $i > 1$. This proves the following result.

\begin{proposition}
	\label{prop:tensor-algebra-is-Koszul} If $R$ is a PID, then the tensor algebra $\rT(\reg_1)$ is Koszul. 
\end{proposition}

The Koszul resolution for exterior algebras is well-known. In particular, the Koszul resolution for the apartment monoid $\bA = \lw(\triv_1)$ is given by the following: \[ \Sym^{\ast}(\triv_1) \otimes \bA \to \bk \to 0.  \] This is equivalent to $\Tor_{\ast}^{\bA}(\bk, \bk) =  \Sym^{\ast}(\triv_1)$.  Let $M$ be an $\bA$-module. Since $\Sym^\ast(\triv_1) \otimes \bA$ is a flat $\bA$-resolution of $\bk$, one can calculate $\Tor_{i}^{\bA}(\bk, M)$ using the Koszul resolution. In other words, we have \[\Tor_{i}^{\bA}(\bk, M) = \rH_i(\Sym^{\ast}(\triv_1)  \otimes M). \] We will need the following calculation later.

\begin{lemma}
	\label{lem:tor-calculation}
	Let $K= \Fq$ be a finite field of size $q$, and let $\bA$ and $\St$ be the apartment and the Steinberg monoids in $\Mod_{\VB_{\Fq}}$. If $\bk$ is a field, then \[\dim_{\bk} \Tor_{2}^{\bA}(\bk, \St)(\Fq^4) \le \frac{(q^4-1)(q^3-1)q^6}{2(q-1)^2}. \]
\end{lemma}
\begin{proof} By the previous paragraph, we have \[ \dim_{\bk} \Tor_{2}^{\bA}(\bk, \St)(\Fq^4) \le \dim_{\bk} (\Sym^2(\triv_1) \otimes \St)(\Fq^4). \] Now note that \begin{align*}
	\dim_{\bk} \Sym^2(\triv_1) &= \frac{|\GL_2(\bF_q)|}{|\GL_1(\bF_q)|^2 |S_2|} \\
	\dim_{\bk} \St(\bF_q^2) &= q,
	\end{align*} and so \[  \dim_{\bk} (\Sym^2(\triv_1) \otimes \St)(\Fq^4) = \frac{|\GL_4(\bF_q)|}{|\GL_2(\bF_q)|^2} \frac{|\GL_2(\bF_q)|}{|\GL_1(\bF_q)|^2 |S_2|} q = \frac{(q^4-1)(q^3-1)q^6}{2(q-1)^2}. \] This completes the proof. 
\end{proof}

We do not know an explicit description of the Koszul resolution for the Steinberg monoid, but it is possible to compute its dimension.

\begin{proposition}
	Let $K= \Fq$ be a finite field of size $q$, and let $\St$ be the  Steinberg monoid in $\Mod_{\VB_{\Fq}}$. If $\bk$ is a field, then $\dim_{\bk} \Tor_{n}^{\St}(\bk, \bk)  = q^{n^2 - n}$.
\end{proposition}
\begin{proof}
	Let $T_n$ denote  $\dim_{\bk} \Tor_{n}^{\St}(\bk, \bk)$, and let ${n\brack i}_q$ denote the $q$-binomial coefficient. It is clear that $T_0 = 1$, and we have \[ \frac{|\GL_n(\bF_q)|}{|\GL_i(\bF_q)||\GL_{n-i}(\bF_q)|} = {n\brack i}_q q^{i(n-i)}. \] Since $\St$ is Koszul (\autoref{thm:Koszulness-of-Steinberg}), we have a Koszul resolution given by $\St \otimes \Tor_{\ast}^{\St}(\bk, \bk)  \to \bk \to 0$. So $T_n$ satisfy the recursion \[ T_n = \sum_{i=1}^n (-1)^i {n\brack i}_q q^{i(n-i)} q^{\binom{i}{2}} T_{n-i}.   \]  By induction, it suffices to verify that \[q^{n^2 - n} = \sum_{i=1}^n (-1)^i {n\brack i}_q q^{i(n-i)} q^{\binom{i}{2}} q^{(n-i)^2 - (n-i)},  \] which is equivalent, after cancelling $q^{\binom{n}{2}}$ from both sides,  to \[q^{\binom{n}{2}} = \sum_{i=1}^n (-1)^i {n\brack i}_q q^{\binom{n-i}{2}}.  \] This follows from \cite[Lemma~5.9]{vimod}, completing the proof.
\end{proof}	
%\begin{proof}
%	$ \Tor_{n}^{\St}(\bk, \bk)$ is the top dimensional homology of $\ol{\cB}^n_{\ast}(\St)$ and the homology of $\ol{\cB}^n_{\ast}(\St)$ is concentrated in top dimension. Thus we can calculate the rank of this group by calculating the Euler characteristic. This gives \[\dim_{\bk} \Tor_{n}^{\St}(\bk, \bk)  = \sum_{s=1}^{n} (-1)^{n-s} \dim_{\bk} \ol{\cB}^n_{s}(\St) = q^{n^2 - n}. \] This completes the proof. 
%\end{proof}

\begin{question}
Can one give a conceptual description of the Koszul dual of the Steinberg monoid, or explicitly describe the cycles in $\ol{\cB}^n_{n}(\St)$? 
\end{question}

\section{Finiteness properties of resolutions of the Steinberg module}
\label{sec:sharbly}

In this section, we study the groups $\Tor_n^{\bA}(\St,\bk)$.

%Then we use it to show finiteness properties of known resolutions of the Steinberg module.

%In this section we study the syzygies of the Steinberg monoid as a module over the monoid. 

\subsection{Presentation of the Steinberg monoid as a module over the apartment monoid}

We reinterpret the Bykovskii presentation \cite{Byk} and the presentation appearing in \autoref{fieldPresentation} as presentations of $\St$ as an $\bA$-module and use this to prove the following. 

\begin{proposition}
	\label{prop:tor-0-1}
	For $R=\Z$ or a field, we have $\deg \Tor^{\bA}_{i}(\bk, \St) = 2i$ for $i = 0,1$.
\end{proposition}

\begin{proof}
The Bykovskii presentation \cite{Byk} (here \autoref{presentationZ}) in the case of the integers  and \autoref{fieldPresentation} in the case of fields 
implies we have an exact sequence \[\bA \otimes M \to \bA \to \St \to 0  \] of $\bA$-modules where $M$ is a $\VB$-module supported only in degree 2 and is given by \[M(X) = \langle[v_1, v_2] - [v_0, v_2] + [v_0,v_1] \mid v_0 = v_1 + v_2 \rangle \subset \bA(X) \] where $X$ is a free $R$-module of rank $2$. The assertion is immediate from this partial  resolution.
\end{proof}

%The same statement is true for $R$ a field. In fact, in \autoref{thm:shably-main}, we prove a generalization to all values of $i$. 

\subsection{Higher syzygies}

\label{sec:proof-sharbly-main} 

In the previous subsection, we showed that for $R=\Z$ or a field, we have $\deg \Tor^{\bA}_{i}(\bk, \St) = 2i$ for $i = 0,1$. In this subsection, we prove that if $R$ is a field, we may drop the restriction on $i$.  Fix a field $K$. Throughout this subsection we shall assume that $R = K$. Let $\St$ and $\bA$ be the  Steinberg and the apartment monoids in $\Mod_{\VB} = \Mod_{\VB_K}$, respectively. The main theorem of this section is the following.

 %It is clear from the definition that $\St$ is generated, as an $\bA$-module, in degree $0$ and related in degree $\le 2$. More precisely, by \autoref{prop:tor-0-1} we have $\deg \Tor_0^{\bA}(\bk, \St) = 0$ and $\deg \Tor_1^{\bA}(\bk, \St) =2$.  

 %One of our main theorem is the following (which is proven in \autoref{sec:proof-sharbly-main}):

\begin{theorem}
	\label{thm:shably-main} For $K$ a field, 
 $\deg \Tor_i^{\bA}(\bk, \St) \le 2i$ for $i \ge 0$.
\end{theorem}

%Recall that in the case of the integers, we showed the analogous statement for $i=0$ and $1$ in \autoref{prop:tor-0-1}. 

Basic properties of $\Tor_{\ast}$ and the theorem above implies the following:

\begin{corollary}[Existence of an improved resolution]
	\label{cor:existence-of-resolution}
	There is a resolution of the form $\bA \otimes V_{\ast} \to \St \to 0$ where $V_i$ is a $\VB_K$-module supported in degrees $\le 2i $.
\end{corollary}

%\jeremy{Cut or sketch a proof? I skimmed [CE] but I could not find an analogous statement.}

%Note that \autoref{thm:church--putman--By} implies that $\St(K^n)$ is isomorphic to $\oSt(K^n)\otimes_{\bZ}\bk$ where $\oSt(K^n)$ is the ordinary Steinberg module of $\GL_n(K)$. 

It would be interesting to find an explicit resolution with the above properties. Below we prove a result on surjections of Koszul monoids (\autoref{prop:map-of-Koszul-algebras}) and combine it with the Koszulness results from the previous section to prove \autoref{thm:shably-main}.

%\begin{remark}
%	\label{rem:tensor1}
%	Let $A = \rT(E')$ be the monoid defined in \autoref{thm:tensor-algebra-is-Koszul}. We have a natural surjection $A \to \St$ of Koszul monoids (obtained by composing the surjections $A \to \bA$ and $\bA \to \St$). \autoref{prop:map-of-Koszul-algebras} applies equally well to this surjection and we obtain $\deg \Tor_i^{A}(\bk, \St) \le 2i$ for $i \ge 0$. 
%\end{remark}

\begin{lemma}
	\label{lem:devisage}
	Suppose $A$ is a Koszul monoid in $(\Mod_{\VB}, \otimes)$. If $M$ is an $A$-module supported in degrees $\le d$, then  $\deg \Tor^A_i(M,\bk) \le i + d$.
\end{lemma}

\begin{proof}
 Let $A \otimes V_{\ast} \to \bk \to 0$ be the Koszul resolution. In particular, $V_i$ is concentrated in degree $i$. Now note that $\Tor^A_{i}(M,\bk) = \rH_i(M \otimes V_{\ast})$. Since $V_i$ is concentrated in degree $i$ and $M$ is supported in degrees $\le d$, we see that $M \otimes V_{\ast}$ is supported in degrees $\le (i + d)$. This finishes the proof.
\end{proof}

\begin{lemma} 
	\label{lem:B-action}
	Let $A \to B \to \bk$ be surjections of (skew) commutative monoids in $(\Mod_{\VB}, \otimes)$. Assume that the kernels $A_+$ and $B_+$ of the surjections $A \to \bk $ and $B \to \bk$ are supported in degrees $> 0$. Then \[\Tor^A_*(\bk,B)\otimes_B \bk = \Tor^A_*(\bk,B).\]
\end{lemma}
\begin{proof}
	Let $P_* \to \bk$ and $Q_{\ast} \to B$ be free $A$-resolutions. Then $ \Tor^A_i(\bk,B) = \rH_i( \bk \otimes_A Q_{\ast})$. Since  $A_+$ acts trivially on $ \bk \otimes_A Q_i = Q_i/(A_{+}Q_i)$, we conclude that $A_+$ acts trivially on  $\Tor^A_i(\bk,B)$.  However, $A_+$ acts via its image $B_+$ on $\Tor^A_i(\bk,B) = \rH_i( P_{\ast}\otimes_A B)$. Thus $B_+$ acts trivially on $\Tor^A_i(\bk,B)$. The assertion is immediate from this.
\end{proof}

The following result is inspired by Chardin--Symonds \cite[\S 5]{CS}.

\begin{proposition}
	\label{prop:map-of-Koszul-algebras} 	Let $A \to B \to \bk$ be surjections of (skew) commutative monoids in $(\Mod_{\VB}, \otimes)$. Assume that the kernels $A_+$ and $B_+$ of the surjections $A \to \bk $ and $B \to \bk$ are supported in degrees $> 0$. If $A$ and $B$ are Koszul, then \[ \deg \Tor_{i}^A( \bk,B) \le 2i \qquad \text{for all } i \ge 0. \]
\end{proposition}
\begin{proof}
	Denote $\deg \Tor_{i}^{A}(\bk,B)$ by $s_i$. We prove by induction on $i$ that $s_i \le 2i$. The base case $i=0$ is trivial. Let $i > 0$. Consider the base change spectral sequence
	\[E^2_{ab} = \Tor^{B}_a(\Tor^{A}_b(\bk,B),\bk) \implies \Tor^{A}_{a+b}(\bk,\bk).\] Let $t_i^{A}$ denote $\deg \Tor^{A}_i(\bk, \bk)$. By Koszulness of $A$, we know that $t_i^{A} \le  i$. Now suppose $b < i$. By induction, $s_b \le 2b$. By \autoref{lem:devisage} and Koszulness of $B$, we have \[\deg E^2_{ab} \le a + s_b \le a + 2b. \] The spectral sequence now implies that \[ \deg E^2_{0, i} \le \max(t_i^{A}, \max_{a+b = i +1, b < i} \deg E^2_{ab} ) \le  2i. \]  Now note that $E^2_{0, i} =  \Tor^{B}_0(\Tor^{A}_i(\bk,B),\bk) = \Tor^{A}_i(\bk,B)$ by \autoref{lem:B-action}. Thus $s_i \le 2i$, completing the proof.
\end{proof}

\begin{proof}[Proof of \autoref{thm:shably-main}]
The theorem follows immediately from the previous proposition and Koszulness of the apartment and the Steinberg monoids (\autoref{thm:apartment-monoid-is-Koszul} and \autoref{thm:Koszulness-of-Steinberg}). 
\end{proof}

\section{Codimension one cohomology of the level 3 congruence subgroup}
\label{sec:congruence}

Let $\Gamma_n(p)$ be the level $p$ congruence subgroup of $\SL_n(\bZ)$. In this section, we show that the integral codimension-one cohomology of $\Gamma_n(3)$ is representation stable in the sense that the sequence \[\{\rH^{\binom{n}{2} -1}(\Gamma_n(3))\}_{n \ge 0}\] of representations is generated in degrees $\le 4$. In other words, the $\GL_n(\bF_3)$-equivariant map \[ \Ind_{\GL_4(\bF_3)}^{\GL_n(\bF_3)} \rH^{\binom{4}{2} -1}(\Gamma_4(3)) \to \rH^{\binom{n}{2} -1}(\Gamma_n(3)) \] is surjective for $n \ge 4$. We also bound the dimensions of these groups with field coefficients. By Borel--Serre duality \cite{BoSe}, we have \[\rH^{\binom{n}{2} - i}(\Gamma_n(3) ;\bk) = \rH_i(\Gamma_n(3); \oSt_n(\bZ)\otimes_{\bZ}  \bk ) \] for every coefficient ring $\bk$. Thus, it suffices to prove the following representation stability result.

\begin{theorem}
	\label{thm:main-congruence}
 %$\rH_1(\Gamma_n(3);\oSt_n\otimes_{\bZ} \bk)$ is representation stable, and t
 
 The sequence \[\{ \rH_1(\Gamma_n(3); \oSt_n(\bZ)\otimes_{\bZ} \bk)\}_{n \ge 0}\] is generated in degrees $\le 4$. Moreover, if $\bk$ is a field and $n \ge 4$ then we have \[\dim_{\bk} \rH_1(\Gamma_n(3); \oSt_n(\bZ)\otimes_{\bZ} \bk) \le \frac{3^{\binom{n-4}{2}}| \GL_n(\bF_3)| }{| \GL_{n-4}(\bF_3)||\GL_4(\bF_3)|} \dim_{\bk} \rH_1(\Gamma_4(3); \oSt_4(\bZ)\otimes_{\bZ} \bk). \]
\end{theorem}

We begin by explaining our setup and constructing a spectral sequence in \autoref{sec:spectral-sequence}. Our setup is quite similar to that of \cite{Dj} and that of \cite{PS}, but we believe that our spectral sequence is new.  In \autoref{sec:p-is-three}, we use  \autoref{thm:shably-main}
on resolutions of the Steinberg monoid and the spectral sequence of \autoref{thm:spectral sequence} to prove \autoref{thm:main-congruence} above. We also obtain a rough upper bound on $\dim_{\bk} \rH_1(\Gamma_4(3); \oSt_4(\bZ)\otimes_{\bZ} \bk)$ to obtain the following corollary.

\begin{corollary} \label{cor:dim-bound-5} Suppose $\bk$ is a field and assume that $n \ge 4$. Then
\[\dim_{\bk} \rH_1(\Gamma_n(3); \oSt_n(\bZ)\otimes_{\bZ} \bk) \le \frac{3^{\binom{n-4}{2}}| \GL_n(\bF_3)| }{| \GL_{n-4}(\bF_3)||\GL_4(\bF_3)|}227340. \]
\end{corollary}
%Finally, in \autoref{sec:stability-conjectures} we state a stability conjecture which is a refined version of a conjecture due to Church--Farb--Putman  \cite[Page 7]{CFPconj}. Our \autoref{thm:main-congruence} is a special case of this conjecture.

\subsection{A spectral sequence}
\label{sec:spectral-sequence} 
Given a groupoid $\cC$, $\Aut_{\cC}$ defines a functor from $\cC$ to the category $\mathbf{Grp}$ of groups. Using this, we define $\GL \colon \VB \to \mathbf{Grp}$ given by $\GL(X) = \Aut_{\VB }(X)$. Moreover, it has the following properties:  \begin{itemize}
	\item The following diagram commutes (also see \cite[Definition~1.1]{Dj}): 	\[
		\begin{tikzcd}[column sep = large]
			X_1 \arrow[r, "f"] \arrow[d, "g"]
			& X_2 \arrow[d, "f_{\ast}(g)"] \\
			X_1 \arrow[r, "f"]
			& X_2
		\end{tikzcd} \] 
\item There is a natural inclusion $\GL(X_1) \times \GL(X_2) \to \GL(X_1 \oplus X_2)$ which is functorial in $X_1$ and $X_2$.
\end{itemize}

\begin{definition}
	\label{def:strong-subfunctor}
Let $\cG \subset \GL$ be a sub-functor, that is, a functor $\cG \colon \VB \to \mathbf{Grp}$ together with a natural transformation $\iota \colon \cG \to \GL$ such that $\iota(X) \colon \cG(X) \to \GL(X)$ is an inclusion for each $X$. It is easy to check that $\cG(X)$ must be a normal subgroup of $\GL(X)$ for each $X$ under this inclusion. We call $\cG$ a {\bf strong sub-functor} if the following condition holds \[ \cG(X_1 \oplus X_2) \cap (\GL(X_1) \times \GL(X_2)) = \cG(X_1) \times \cG(X_2).\] For a strong sub-functor $\cG$, we define $\VB_{\bZ}/\cG$ to be the category with the same objects as $\VB_{\bZ}$ and whose morphisms are $\Hom_{\VB_{\bZ}/\cG}(X_1, X_2) =  \cG(X_2) \backslash \Hom_{\VB_{\bZ}}(X_1, X_2)$. There is a natural full and essentially surjective strong symmetric monoidal  functor $\Pi_{\cG} \colon \VB_{\bZ} \to \VB_{\bZ}/\cG$. 
\end{definition}

\begin{ex}
	\label{exs:gamma-p}
 Clearly, $\GL$ itself is strong. For a more sophisticated example, consider the natural strong symmetric monoidal functor $\Phi \colon \VB_{\bZ}  \to \pVB$ given by $X \mapsto \bF_p \otimes_{\bZ} X$. Then   \begin{align}
 \label{eq:gamma-p} \Gamma(p)(X) = \ker(\Aut_{\VB_{\bZ} }(X) \to \Aut_{\VB}(\Phi(X)) ) 
\end{align} defines a functor $\Gamma(p) \colon \VB_{\bZ} \to \mathbf{Grp}$ which is strong. We shall use the usual abbreviation $\Gamma_n(p)$ for $\Gamma(p)(\bZ^n)$. The functor $\Phi$ factors through $\Pi_{\Gamma(p)}$. 	\begin{figure}[h]
	\begin{tikzcd}
		\VB_{\bZ} \arrow[rr, "\Phi"] \arrow[rd, "\Pi_{\Gamma(p)}", swap]
		& & \pVB  \\
		& \VB_{\bZ}/\Gamma(p) \arrow[ru, "\iota", swap]
	\end{tikzcd}
\end{figure}  The strong monoidal functor $\iota \colon \VB_{\bZ} / \Gamma(p) \to \pVB$ as in the diagram is essentially surjective. Moreover, 
\begin{enumerate}
\item $\iota$ is  full if $p \le 3$.
\item $\iota$ is  faithful if $p \ge 3$.
\end{enumerate}
In particular, $\iota$ is a monoidal equivalence of categories if $p=3$. To see this when $p=3$, just note that we have $\bZ^{\times} = \bF_3^{\times} = \{\pm 1 \}$. Since $\GL_n(\bZ)$ is generated by elementary matrices and their reductions modulo $3$ generate $\GL_n(\bF_3)$, it follows that there is a natural isomorphism $\Gamma_n(3) \backslash \GL_n(\bZ) \cong \GL_n(\bF_3)$. Thus $\iota$ is a monoidal equivalence of categories. In fact, more is true when $p =3$. We claim that the map $\iota^{\ast} \colon \Mod_{\VB_{\bF_3}} \to \Mod_{\VB_{\bZ}/\Gamma(3)}$ induced by $\iota$ is a monoidal equivalence. Note that $\iota$ induces the following isomorphism of transformations 	\begin{figure}[h]
	\begin{tikzcd}[column sep = large]
		\Aut_{\VB_{\bZ}/\Gamma(3)}(X_1) \times 	\Aut_{\VB_{\bZ}/\Gamma(3)}(X_2) 	 \arrow[r] \arrow[d, "\iota"]
		& \Aut_{\VB_{\bZ}/\Gamma(3)}(X_1 \oplus X_2) \arrow[d, "\iota"] \\
		\Aut_{\VB_{\bF_3}}(\iota(X_1)) \times 	\Aut_{\VB_{\bF_3}}(\iota(X_2))  \arrow[r]
		& 	\Aut_{\VB_{\bF_3}}(\iota(X_1) \oplus \iota(X_2))
	\end{tikzcd}
\end{figure} \\ which can easily be verified.  Moreover, we have a natural isomorphism (which is equivalent to the strongness of $\Gamma(3)$, and can be easily seen via an application of the second isomorphism theorem) \[ \iota \colon   \Gamma(3)(X) \backslash \GL(X)/ (\GL(X_1) \times \GL(X_2)) \to \GL(\iota(X))/(\GL(\iota(X_1)) \times \GL(\iota(X_2))) \] where $X = X_1 \oplus X_2$. This shows that the left Kan extensions as in \autoref{rem:djament's-tensor} agrees for the two categories. In particular, $\iota^{\ast} \colon \Mod_{\VB_{\bF_3}} \to \Mod_{\VB_{\bZ}/\Gamma(3)}$ is a monoidal equivalence.

For an odd prime $p$, the category $\VB_{\bZ} / \Gamma(p) $ is equivalent to the category $\VB_{\bF_p}^\pm$ which has appeared in the work of Putman--Sam on congruence subgroups \cite{PS}. An orientation on a finite dimensional $\bF_p$-vector space $X$ of dimension $n$ is a choice of generator of $\bigwedge^{n} X$ defined up to multiplication by $\pm 1$ (we warn the reader that this is different from the usual definition of orientation which is just an isomorphism $\bF_p \to \bigwedge^{n} X$). Let $\VB_{\bF_p}^\pm$ be the category with objects given by finite dimensional oriented $\bF_p$-vector spaces and with morphisms given by isomorphisms preserving the orientation. The automorphism groups in this category are isomorphic to $\GL_n^\pm(\bF_p)$, the  subgroup of $\GL_n(\bF_p)$ consisting of matrices with determinant $\pm 1$. Using strongness (\autoref{def:strong-subfunctor}), as in the case $p=3$, one can check that $\Mod_{\VB_{\bZ}/\Gamma(p)}$ is monoidally equivalent to $\Mod_{\VB_{\bF_p}^{\pm}}$.
\end{ex}

From now on, we will assume that $\cG$ is an arbitrary strong sub-functor of $\GL$. The reader is advised to keep in mind the sub-functors in the example above as we will focus only on these later.

\begin{proposition} 
	\label{prop:H-is-monoidal}
	We have the following: \begin{enumerate}
		\item The pullback functor $\Pi_{\cG}^{\ast} \colon \Mod_{\VB_{\bZ}/\cG} \to \Mod_{\VB_{\bZ}}$ has a  left adjoint $\rH_0(\cG; -)$ given by \[\rH_0(\cG; M)(X) = \rH_0(\cG(X); M(X)). \]
		\item $\rH_0(\cG; -)$ is strong monoidal. 
	\end{enumerate}
\end{proposition}
\begin{proof}
We note that our tensor product is the same as Djament's; see \autoref{rem:djament's-tensor}. Part (a) is proven in \cite[Proposition~1.3, 1.4]{Dj}. Now we prove Part (b). In \cite[\S 1.5]{Dj}, it is shown that $\rH_0(\cG; -)$ is lax monoidal. We claim $\rH_0(\cG; -)$ is in fact strong monoidal which we verify as follows: \begin{align*}
\rH_0(\cG; M \otimes N)(X) &= \rH_0(\cG(X); (M \otimes N)(X)) \\
&= \rH_0\left(\cG(X); \bigoplus_{X_1 \oplus X_2 = X \text{ in } \VB_{\bZ}} M(X_1) \otimes_{\bk} N(X_2)\right) \\
&= \rH_0\left(\cG(X); \bigoplus_{X_1 \oplus X_2 = X \text{ in } \VB_{\bZ}/\cG} \Ind_{\cG(X) \cap (\GL(X_1) \times \GL(X_2))}^{\cG(X)} M(X_1) \otimes_{\bk} N(X_2)\right) \\
&= \bigoplus_{X_1 \oplus X_2 = X \text{ in } \VB_{\bZ}/\cG} \rH_0(\cG(X);  \Ind_{\cG(X_1) \times \cG(X_2)}^{\cG(X)} M(X_1) \otimes_{\bk} N(X_2)) \\
& = \bigoplus_{X_1 \oplus X_2 = X \text{ in } \VB_{\bZ}/\cG} \rH_0(\cG(X_1) \times \cG(X_2);   M(X_1) \otimes_{\bk} N(X_2)) \\
& = (\rH_0(\cG; M ) \otimes \rH_0(\cG;  N))(X).
\end{align*} where the third equality follows from the Mackey decomposition theorem, the fourth equality follows from the fact that $\cG$ is strong, and the last equality follows from the K\"unneth formula for group homology. This completes the proof.
\end{proof}

Since $\rH_0(\cG; -)$ is monoidal, it takes a monoid $A$ in $\Mod_{\VB_{\bZ}}$ to a monoid $B \coloneq \rH_0(\cG; A)$  in $\Mod_{\VB_{\bZ}/\cG}$. Moreover, $\rH_0(\cG; -)$ restricts to a functor  \[ \rH_0^A(\cG; - ) \colon \Mod_{A} \to \Mod_B.\] The functors $\rH_0(\cG; -)$ and $\rH_0^A(\cG; -) $ are right exact. We denote their left derived functors by $\rH_{\ast}(\cG; -)$ and $\rH^A_{\ast}(\cG; -)$ respectively.

\begin{remark}
	\label{rem:homology-warning}
	A $\VB_{\bZ}$-module $P$ is projective if and only if $P(X)$ is projective as a $\bk[\GL(X)]$-module for all $X$. This shows that  $\rH_i(\cG; M)(X) \cong \rH_i(\cG(X); M(X))$. We warn the readers that, for a general monoid $A$, we do not have any natural relationship between $\rH^A_i(\cG; M)(X)$ and $\rH_i(\cG(X); M(X))$ for $i > 0$. This is because a projective module in $\Mod_{A}$ may not be projective (or even flat) as a $\VB_{\bZ}$-module. See \autoref{lem:char0-is-good} for more on this. 
\end{remark}

\begin{proposition}
	\label{prop:tensor-commute}
	$\rH_0(\Gamma(p); -)$ commutes with tensor, symmetric and exterior algebras. In other words, if $M$ is a $\VB_{\bZ}$-module, then we have \begin{enumerate}
		\item $\rH_0(\Gamma(p); \rT(M)) = \rT(\rH_0(\Gamma(p); M))$ where $\rT$ is the tensor algebra.
		\item $\rH_0(\Gamma(p); \Sym(M)) = \Sym(\rH_0(\Gamma(p); M))$.
			\item $\rH_0(\Gamma(p); \lw(M)) = \lw(\rH_0(\Gamma(p); M))$.
	\end{enumerate}
\end{proposition}
\begin{proof}
Part (a) follows immediately from the fact that $\rH_0(\cG; -)$ is strong monoidal. Moreover, since the action of the symmetric group on tensors commutes with the action of the congruence subgroup, we see that parts (b) and (c) hold as well.
\end{proof}

\begin{proposition}
Let $A$ be an augmented monoid in  $\Mod_{\VB_{\bZ}}$. Set $B = \rH_0(\cG; A)$.	Then the following diagram commutes. \begin{figure}[h]
		\begin{tikzcd}[column sep = large]
			\Mod_A \arrow[r, "\rH_0^A(\cG; -)"] \arrow[d, "\bk \otimes_{A} -"]
			& \Mod_B \arrow[d, "\bk \otimes_B - "] \\
			\Mod_{\VB_{\bZ}} \arrow[r, "\rH_0(\cG; -)"]
			& \Mod_{\VB_{\bZ}/ \cG}
		\end{tikzcd}
	\end{figure} 
\end{proposition}
\begin{proof}
Let $I_A$ and $I_B$ be the right adjoints of $ \bk \otimes_A -$ and $ \bk \otimes_B -$  respectively. By our assumption on $A$, $I_A(M)$ is the same as $M$ regarded as an $A$-module via the augmentation map $A \to \bk$. A similar statement is true for $I_B$. This shows that $\rH_0^A(\cG; -) \circ I_A = I_B \circ \rH_0(\cG; -)$. 

Let $\epsilon \colon \id \to I_A \circ (\bk\otimes_A -) $ be the unit of the adjunction. By the previous paragraph, we have a natural transformation $\rH_0^A(\cG; -) \to I_B \circ \rH_0(\cG; -) \circ (\bk\otimes_A -) $ obtained by composing $\rH_0^A(\cG; -)$ with the unit $\epsilon$. By adjunction, there exists a natural transformation \[(\bk\otimes_B -) \circ  \rH_0^A(\cG; -) \to  \rH_0(\cG; -) \circ (\bk\otimes_A -)\]  which we claim is an isomorphism. Since all the functors involved are right exact, it suffices to prove the claim for objects of the form $A \otimes V$ (projective objects are of this form). We now check this as follows: \begin{align*}
(\bk\otimes_B -) \circ  \rH_0^A(\cG; - )(A \otimes V) &= \bk \otimes_B \rH_0^A(\cG; A \otimes V) \\
& = \bk \otimes_B (\rH_0(\cG; A) \otimes    \rH_0(\cG; V)) \qquad \text{since $\rH_0(\cG; -)$ is monoidal} \\
& = \bk \otimes_B (B \otimes \rH_0(\cG; V)) \\
& = \rH_0(\cG; V) =  \rH_0(\cG; \bk \otimes_A (A \otimes V)) \\
& =  \rH_0(\cG; -) \circ (\bk\otimes_A -) (A \otimes V).
\end{align*} This completes the proof.
\end{proof}

Note that $\Mod_{\VB_{\bZ}}$ has enough projectives -- $P$ is a projective $\VB_{\bZ}$-module if and only if $P(\bZ^n)$ is a projective $\bZ[\GL_n(\bZ)]$-module for each $n \ge 0$. In particular, we can define left-derived functors $\rL_i F$ of a given right exact functor $F$ defined on $\Mod_{\VB_{\bZ}}$.

\begin{theorem}
	\label{thm:spectral sequence}
	Let $A$ be an augmented monoid in  $\Mod_{\VB_{\bZ}}$. Set $B = \rH_0(\cG; A)$. Then there exists a spectral sequence \[E^2_{ab}(-) = \Tor^B_a(\bk, \rH^A_b(\cG; -))  \implies \rL_{a+b}(\rH_0(\cG; -) \circ (\bk\otimes_A -) ). \]  In particular, if $M$ is an $A$-module such that $\max_{ i \le n}(\deg \Tor^A_i(\bk, M) ) \le d$, then $E^{\infty}_{ab}(M)$ is supported in degrees $\le d$ for all $a+b \le n$.  
\end{theorem}
\begin{proof} We first verify that $\rH_0^A(\cG; -)$ preserves projectives. It is an easy fact that every projective $A$-module is of the form $A\otimes V$ where $V$ is a projective $\VB$-module, and the same holds for $B$-modules; see \cite[Proposition~3.2]{vimod} for a proof in a similar setting. Since  $\rH_0(\cG; -)$ is monoidal, we see that $\rH_0^A(\cG; A \otimes V) = B \otimes \rH_0(\cG; V)$. By definition, $\rH_0(\cG; -)$ is a left adjoint of an exact functor $\Pi^{\ast}_{\cG}$ and so it preserves projectives. Since $\rH_0^A(\cG; A \otimes V) = B \otimes \rH_0(\cG; V)$, we conclude that $\rH_0^A(\cG; -)$ preserves projectives. This verifies our claim. 

Note that $\Mod_B$ has enough projectives. We conclude that there exists a Grothendieck spectral sequence \[E^2_{ab}(-) = \Tor^B_a(\bk, \rH^A_b(\cG; -))  \implies \rL_{a+b}((\bk\otimes_B -) \circ  \rH_0^A(\cG; - ) ). \] By the previous proposition $(\bk\otimes_B -) \circ  \rH_0^A(\cG; - ) = \rH_0(\cG; -) \circ (\bk\otimes_A -) $, and so the first assertion follows.
	
	It is clear from the last paragraph that $\bk \otimes_A - $ preserves projectives. Also, $\Mod_{\VB}$ has enough projectives. This shows that we have another  Grothendieck spectral sequence \[ \rH_a(\cG; \Tor^A_b(\bk, - )) \implies \rL_{a+b}(\rH_0(\cG; -) \circ (\bk\otimes_A -) ). \] If $\max_{ i \le n}(\deg \Tor^A_i(\bk, M) ) \le d$, then we have $\deg  \rL_{a+b}(\rH_0(\cG; -) \circ (\bk\otimes_A -) ) \le d$ for all $a+b \le n$ (also see \autoref{rem:homology-warning}). This proves the second assertion.
\end{proof}

The following lemma shows that our spectral sequence is particularly useful when $A$ is the exterior algebra (see \autoref{rem:homology-warning}). Similar results hold for tensor and symmetric algebras.

\begin{lemma}
	\label{lem:char0-is-good}
	Assume that $\cG(X)$ is torsion-free for all $X$ (this happens, for example, when $p$ is an odd prime and $\cG = \Gamma(p)$). Let $\bA = \lw(\triv_1)$ be the apartment monoid in $\Mod_{\VB_{\bZ}}$. Then for every $\bA$-module $M$ and every $X \in \VB_{\bZ}$, we have an isomorphism  \[\rH^{\bA}_i(\cG; M)(X) \cong \rH_i(\cG(X); M(X)) \] of $\GL(X)$-modules  for each $i \ge 0$. 
\end{lemma}
\begin{proof} Note that $\rH^{\bA}_i(\cG; M)$  is calculated by first taking an $\bA$-module projective resolution $P_{\ast} \to M \to 0$ of $M$, applying $\cG$-coinvariants to the resolution, and then taking $i$th homology of the resulting complex $(P_{\ast})_{\cG}$. Thus to prove the isomorphism in the assertion, it suffices to show that a projective $\bA$-module  is $\rH_0(\cG;  -)$-acyclic when regarded as a $\VB_{\bZ}$-module (recall that $\Mod_{\VB_{\bZ}}$ is equivalent to the product category $\prod_{\ge 0}\Mod_{\bk[\GL_n(\bZ)]}$, and so a module is projective if and only if it is pointwise projective). Let $P$ be a projective $\bA$-module. Then $P$ is of the form $\bA \otimes V$ for some projective $\VB_{\bZ}$-module $V$. We claim that $\rH_i(\cG(X); P(X)) =0$ for all $X$ and all $i >0$. Note that we have:  \begin{align*}
	\rH_i(\cG(X); P(X)) &= \rH_i(\cG; \bA \otimes V)(X)  \\
	&= (\rH_i(\cG; \bA ) \otimes \rH_0(\cG;  V)) (X) 
	\end{align*} where the first equality follows from  \autoref{rem:homology-warning}. The second equality is obtained by taking a projective resolution $Q_{\ast} \to \bA$ in $\Mod_{\VB_{\bZ}}$, and then calculating $\rH_i(\cG; \bA \otimes V)$ using the projective resolution $Q_{\ast} \otimes V \to \bA \otimes V$ and monoidality of $\rH_0(\cG;  -)$. Thus, it suffices to show that $\rH_i(\cG; \bA) = 0$ for $i >0$. We have that $\bA(\Z^n)$ is isomorphic to $\rH_0(\Z/2 \wr S_n;  \bk[ \GL_n(\bZ)])$ where $\Z/2 \wr S_n$ acts on the right as in the definition of the apartment monoid.  Since 	$\Z/2 \wr S_n$ is a finite group,  the stabilizer of $\cG(\Z^n)$ acting on  $ \bA(\Z^n)$ must have finite order. Since $\cG(\Z^n) $ is torsion free, this stabilizer is trivial. Thus, $\bA(X)$ is a free $\bk[\cG(X)]$-module. This shows that $\rH_i(\cG; \bA)(X) = 0$ for $i >0$. This verifies our claim, and the proof is complete.
\end{proof}

Often a $\rH_0(\cG; A)$ structure on $\rH_i(\cG; A)$ is defined directly using the K\"unneth map and the Shapiro lemma instead of relying on any isomorphisms as in the previous lemma. We now describe it and compare it with our setup. For this, we define an enrichment \[\rH_{\ast}(\cG ; -) \colon \Mod_{\VB_{\bZ}} \to \Mod_{\VB_{\bZ}/\cG}^{\gr}\] of $\rH_0(\cG ; -)$, where $\Mod_{\VB_{\bZ}/\cG}^{\gr}$ is the category of graded $\VB_{\bZ}/\cG$-modules, by \[\rH_{\ast}(\cG ; -) = \bigoplus_{t \ge 0} \rH_t(\cG ; - ). \] Note $\rH_{\ast}(\cG ; M)$ is bigraded with one grading coming from the $\VB_{\bZ}/\cG$-module structure and the other grading coming from homological degree. The tensor product on $\Mod_{\VB_{\bZ}/\cG}^{\gr}$ is the usual convolutional tensor product with respect to both of the gradings.

\begin{proposition}  $\rH_{\ast}(\cG ; - )$ is a lax monoidal functor. \end{proposition} 
\begin{proof} The natural map $\rH_{\ast}(\cG ; M) \otimes \rH_{\ast}(\cG ; N) \to \rH_{\ast}(\cG ; M \otimes N)$  given by composing the K\"unneth map with the Shapiro isomorphism provides the required lax monoidal structure. The details are very similar to the ones in \autoref{prop:H-is-monoidal} Part (b). We get lax monoidality instead of strong monoidality because the K\"unneth map is not an isomorphism in general.
\end{proof}

The proposition above shows that if $A$ is a monoid and $M$ is an $A$-module, then $\rH_{\ast}(\cG ; A )$ is a monoid and $\rH_{\ast}(\cG ; M )$ is an $\rH_{\ast}(\cG ; A )$-module. Note that $\rH_{0}(\cG ; A )$ is naturally a sub-monoid of $\rH_{\ast}(\cG ; A )$, and so $\rH_{\ast}(\cG ; M )$ is an $\rH_0(\cG ; A )$-module. Let \[\rH_{\ast}^A(\cG ; M) = \bigoplus_{t \ge 0} \rH_t^A(\cG ; M). \] By definition, $\rH_{\ast}^A(\cG ; M)$ is an $\rH_0(\cG ; A )$-module. We now relate the $\rH_0(\cG ; A )$-modules $\rH_{\ast}(\cG ; M )$ and $\rH_{\ast}^A(\cG ; M )$. 

\begin{proposition}
	\label{prop:factoring-through}
	Suppose $f \colon A \to B$ is a map of monoids in $\Mod_{\VB_{\bZ}}$ and assume that projective $A$-modules are $\rH_{\ast}(\cG ; -)$-acyclic. Then we have the following: \begin{enumerate}
		\item For any $A$-module $M$, $\rH_{\ast}^A(\cG ; M)$ and $\rH_{\ast}(\cG ; M)$ are isomorphic as $ \rH_0(\cG ; A)$-modules. 
		\item Denote the isomorphism in the previous part by $\phi$. For any $B$-module $M$, the following diagram commutes: \[
			\begin{tikzcd}
				& \rH_{0}(\cG ; A) \otimes \rH^A_{\ast}(\cG ; M) \arrow[rd, "\mathrm{product}"] \arrow[ld, "\mathrm{inclusion} \otimes \phi", swap] 
				 \\
				\rH_{\ast}(\cG ; A) \otimes \rH_{\ast}(\cG ; M) \arrow[rr, "\mathrm{product}"] \arrow[rd, "\rH_{\ast}(\cG ; f) \otimes \id", swap]
				& &   \rH_{\ast}(\cG ; M)  \\
				&  \rH_{\ast}(\cG ; B) \otimes \rH_{\ast}(\cG ; M) \arrow[ru, "\mathrm{product}", swap]
			\end{tikzcd}\]
	\end{enumerate}
\end{proposition}
\begin{proof} Proof of Part (a). It is clear from the $\rH_{\ast}(\cG ; -)$-acyclicity of projective $A$-modules that $\rH_{\ast}^A(\cG ; M)$ and $\rH_{\ast}(\cG ; M)$ are isomorphic as $\Mod_{\VB_{\bZ}/\cG}^{\gr}$-modules. In particular, an $A$-projective resolution $P_{\ast} \to M$ of $M$ can be used to calculate both  $\rH_{\ast}(\cG ; M )$ and $\rH_{\ast}^A(\cG ; M )$. The $ \rH_0(\cG ; A)$ action on $ \rH_0(\cG ; P_i)$ is given by composing the K\"unneth map and the Shapiro isomorphism. Since these two maps are functorial, we see that the action of $ \rH_0(\cG ; A)$ on \[ \rH_i(\rH_0(\cG ; P_\ast)) = \rH_i^A(\cG ; M) \] is given by composing the K\"unneth map and the Shapiro isomorphism. This completes the proof of Part (a).

We now prove Part  (b). The commutativity of the top triangle is exactly Part (a). The bottom triangle commutes because of the proposition above and functoriality of the K\"unneth and the Shapiro maps. This finishes the proof.
\end{proof}

\subsection{Proof of \autoref{thm:main-congruence} (the main theorem on congruence subgroup)}
\label{sec:p-is-three}
We now use our main technical result (\autoref{thm:shably-main}) on the Steinberg monoid as an apartment module and the spectral sequence in \autoref{thm:spectral sequence}
 to prove \autoref{thm:main-congruence} on level 3 congruence subgroup. We first explain how to use our spectral sequence.
 
Let $\bA$ and $\St$ be the apartment and the Steinberg monoids in $\Mod_{\VB_{\bZ}}$. Let $p$ be a prime, and let $\Gamma(p)$ be as described in \eqref{eq:gamma-p}.  Applying \autoref{thm:spectral sequence} to the $\bA$-module $\St$, we obtain the following spectral sequence:
\[E^2_{ab} = \Tor^{\rH_0(\Gamma(p); \bA)}_{a}(\bk, \rH^{\bA}_{b}(\Gamma(p); \St))  \implies \rL_{a+b}(\rH_0(\Gamma(p); -) \circ (\bk\otimes_{\bA} -) )(\St). \] We now simplify this spectral sequence.  Since $ \rH_0(\Gamma(p); -)$ is monoidal, $\rH_0(\Gamma(p); \bA)$ is a monoid in $\Mod_{\VB_{\bZ}/\Gamma(p)}$ and $\rH_0(\Gamma(p); \St)$ is a module over it. For brevity, we set $\bA_{\Gamma(p)} \coloneq \rH_0(\Gamma(p); \bA)$ and $\St_{\Gamma(p)} \coloneq \rH_0(\Gamma(p); \St)$.  \autoref{lem:char0-is-good} tells us that $\rH^{\bA}_{b}(\Gamma(p); -)$ is isomorphic to $\rH_{b}(\Gamma(p); -)$ as $\VB_{\bZ}/\Gamma(p)$-modules, but comes equipped with an action of $\bA_{\Gamma(p)}$. Thus we can drop the superscript $\bA$ without causing any issues. We now obtain the following simplified spectral sequence \begin{align}
\label{eq:simplified-spectral-sequence}
E^2_{ab} = \Tor^{\bA_{\Gamma(p)}}_{a}(\bk, \rH_{b}(\Gamma(p); \St))  \implies \rL_{a+b}(\rH_0(\Gamma(p); -) \circ (\bk\otimes_{\bA} -) )(\St).
\end{align} We need one more ingredient to be able to use this spectral sequence  to prove our main theorem. The following proposition is this ingredient and is precisely the place where we use our main technical result (\autoref{thm:shably-main}) and the assumption that $p=3$. We defer the proof of this proposition until the next subsection and concentrate on using it first.

\begin{proposition}
	\label{prop:bound-p-3}
	We have $\deg \Tor^{\bA_{\Gamma(3)}}_{i}(\bk,  \St_{\Gamma(3)})  \le 2i$ for all $i \ge 0$. Moreover, we have $\dim_{\bk} \St_{\Gamma(3)}(\bZ^{n}) = 3^{\binom{n}{2}}$.
\end{proposition}

The following proposition is a more abstract version of our main theorem. 

\begin{proposition} \label{maintheoremfancy}
	$  \rH_1(\Gamma(3); \St)$ is an $\bA_{\Gamma(3)}$-module generated in degrees $\le 4$. 
\end{proposition}
\begin{proof}
	When $p=3$, the spectral sequence \eqref{eq:simplified-spectral-sequence} becomes
	\[E^2_{ab} = \Tor^{\bA_{\Gamma(3)}}_{a}(\bk, \rH_{b}(\Gamma(3); \St))  \implies \rL_{a+b}(\rH_0(\Gamma(3); -) \circ (\bk\otimes_{\bA} -) )(\St). \]  Note that $\deg \Tor_i^{\bA}(\bk, \St) \le 2i$ for $i =0,1$ (by \autoref{prop:tor-0-1}). Thus by the second assertion of \autoref{thm:spectral sequence}, we see that $E^{\infty}_{ab}$ is supported in degrees $\le 2$ for $a + b \le 1$. Now note that $E^{\infty}_{0,1}$ is the cokernel of the map \[ \Tor^{\bA_{\Gamma(3)}}_2(\bk, \rH_0(\Gamma(3); \St)) = \Tor^{\bA_{\Gamma(3)}}_2(\bk, \St_{\Gamma(3)}) \to \Tor^{\bA_{\Gamma(3)}}_0(\bk, \rH_1(\Gamma(3); \St)). \] By the previous proposition, $\Tor^{\bA_{\Gamma(3)}}_2(\bk, \St_{\Gamma(3)})$ is supported in degrees $\le 4$, and so we conclude that \[
	\deg \Tor^{\bA_{\Gamma(3)}}_0(\bk, \rH_1(\Gamma(3); \St)) \leq 4, \] completing the proof of the proposition.
\end{proof}

\begin{corollary}
	\label{cor:fancy-St}
	$  \rH_1(\Gamma(3); \St)$ is a $\St_{\Gamma(3)}$-module generated in degrees $\le 4$. 
\end{corollary}
\begin{proof}
	This is immediate from the proposition above and \autoref{prop:factoring-through} Part (b).
\end{proof}

\begin{remark}
Bounds on $\deg \Tor_2^{\bA}(\bk, \St)$ could be used together with the arguments of this paper to bound  $ \Tor^{\bA_{\Gamma(3)}}_1(\bk, \rH_1(\Gamma(3); \St))$. Similarly, bounds on $\deg \Tor_i^{\bA}(\bk, \St)$ seem likely to be useful for bounding degrees of the higher syzygies  $ \Tor^{\bA_{\Gamma(3)}}_i(\bk, \rH_j(\Gamma(3); \St))$. However, bounds on $\deg \Tor_i^{\bA}(\bk, \St)$ are not known for $i >1$. 
\end{remark}

\begin{proof}[Proof of \autoref{thm:main-congruence}] 
	 \autoref{maintheoremfancy} implies that  \[ \bigoplus_{\substack{X_1 \oplus X_2 = X \text{ in } \VB_{\bZ}/\Gamma(3), \\ \rank X_2 = 4}} \bA_{\Gamma(3)}(X_1) \otimes_{\bk}  \rH_1(\Gamma(3); \St)(X_2) \to \rH_1(\Gamma(3); \St)(X) \] is surjective for $\rank X \ge 4$. Setting $X =\bZ^n$ for $n \ge 4$, we obtain the following surjection: \[ \Ind_{\Aut_{\VB_{\bZ}/\Gamma(3)}(\bZ^{n-4}) \times \Aut_{\VB_{\bZ}/\Gamma(3)}(\bZ^4)}^{\Aut_{\VB_{\bZ}/\Gamma(3)}(\bZ^n)} \bA_{\Gamma(3)}(\bZ^{n-4}) \otimes_{\bk}  \rH_1(\Gamma(3); \St)(\bZ^4) \to \rH_1(\Gamma(3); \St)(\bZ^n). \] As noted in \autoref{exs:gamma-p}, we have $\Aut_{\VB_{\bZ}/\Gamma_{3}}(\bZ^n) = \GL_{n}(\bF_3)$. Using the monoidal equivalence $\iota^{\ast}$ as in that example, we see that the following map is a surjection: \begin{align}
	\label{eq:induction}
	 \Ind_{\GL_{n-4}(\bF_3) \times \GL_4(\bF_3)}^{\GL_n(\bF_3)} \bA_{\Gamma(3)}(\bZ^{n-4}) \otimes_{\bk}  \rH_1(\Gamma(3); \St)(\bZ^4) \to \rH_1(\Gamma(3); \St)(\bZ^n). \end{align} We also have a natural surjection:  \[\bk[\Aut_{\VB_{\bZ}/\Gamma(3)}(\bZ^{n-4})] = \bk[\GL_{n-4}(\bF_3)] \to \bA_{\Gamma(3)}(\bZ^{n-4}).\] Combining it with \eqref{eq:induction}, we obtain the desired surjection: \[ \Ind_{\GL_4(\bF_3)}^{\GL_n(\bF_3)}   \rH_1(\Gamma(3); \St)(\bZ^4) \to \rH_1(\Gamma(3); \St)(\bZ^n). \] The first assertion now follows from the equality $\rH_1(\Gamma(3); \St)(\bZ^n) = \rH_1(\Gamma_n(3); \St_n(\bZ))$. 
	 
	 For the second assertion, note that by \autoref{cor:fancy-St} we have the following analogue of \eqref{eq:induction}: \begin{align}
	 \Ind_{\GL_{n-4}(\bF_3) \times \GL_4(\bF_3)}^{\GL_n(\bF_3)} \St_{\Gamma(3)}(\bZ^{n-4}) \otimes_{\bk}  \rH_1(\Gamma(3); \St)(\bZ^4) \to \rH_1(\Gamma(3); \St)(\bZ^n). \end{align} The second assertion now follows from the fact that $\dim_{\bk} \St_{\Gamma(3)}(\bZ^{n-4}) = 3^{\binom{n-4}{2}}$.
\end{proof}

We now concentrate on proving the dimension bounds in \autoref{cor:dim-bound-5}.

\begin{lemma}
	\label{lem:dim-bound-3}
	Suppose $\bk$ is a field. We have $\dim_{\bk} \rH_1(\Gamma_3(3); \oSt_3(\bZ)\otimes_{\bZ} \bk)  \le 35$. 
\end{lemma}
\begin{proof}
	By \cite[Theorem~1.4]{LS}, we have $\dim_{\bk} \Hom_{\bZ}(\rH_2(\Gamma_3(3)), \bk ) \le 27 $. Moreover, by \cite[Lemma~12.1]{LS} we have \[\dim_{\bk} \Ext^1_{\bZ}(\rH_2(\Gamma_3(3)), \bk) = \begin{cases}
	8 & \mbox{if } \mathrm{Char}(\bk) = 3\\
	0 & \mbox{otherwise}.
	\end{cases}  \] Thus by the universal coefficient theorem, we have $\dim_{\bk} \rH^2(\Gamma_3(3); \bk)  \le 35$. By Borel--Serre duality, we conclude that $\dim_{\bk} \rH_1(\Gamma_3(3); \oSt_3(\bZ)\otimes_{\bZ} \bk)  \le 35$. 
\end{proof}

The following lemma and \autoref{thm:main-congruence} complete the proof of \autoref{cor:dim-bound-5}.

\begin{lemma}
	\label{lem:dim-bound-4}
	Suppose $\bk$ is a field.  Then we have \[\dim_{\bk} \rH_1(\Gamma_4(3); \oSt_4(\bZ) \otimes_{\bZ} \bk) = \dim_{\bk} \rH_1(\Gamma(3); \St)(\bZ^4) \le 227340.  \]
\end{lemma}
\begin{proof}
	Denote the $\bA_{\Gamma(3)}$-module  $\rH_1(\Gamma(3); \St)$ by $M$, and let $M'$ be the maximal submodule of $M$ generated in degrees $\le 3$. Let $V$ be the degree 4 piece of the  $\VB_{\bZ}/{\Gamma(3)}$-module $\Tor^{\bA_{\Gamma(3)}}_0(\bk, \rH_1(\Gamma(3); \St))$. Then by the definition of  $\Tor^{\bA_{\Gamma(3)}}_0(\bk, -)$, we have an isomorphism \[V(\bZ^4) \cong (M/M')(\bZ^4).  \] This shows that \[\dim_{\bk} M(\bZ^4) = \dim_{\bk} M'(\bZ^4) + \dim_{\bk} V(\bZ^4).\]   The proof of \autoref{maintheoremfancy} shows that the map $\Tor^{\bA_{\Gamma(3)}}_2(\bk, \St_{\Gamma(3)}) \to \Tor^{\bA_{\Gamma(3)}}_0(\bk, \rH_1(\Gamma(3); \St))$ is surjective in degrees 3 and 4. In particular, we have \[\dim_{\bk} V(\bZ^4) \le \dim_{\bk} \Tor^{\bA_{\Gamma(3)}}_2(\bk, \St_{\Gamma(3)})(\bZ^4). \] By \autoref{prop:LS-case3} and \autoref{lem:tor-calculation} (for $q = 3$), we have \[\dim_{\bk} \Tor^{\bA_{\Gamma(3)}}_2(\bk, \St_{\Gamma(3)})(\bZ^4) \le \frac{(3^4-1)(3^3-1)3^6}{8} = 189540. \] By the previous lemma, $\dim_{\bk} M'(\bZ^3) \le 35$. Since $M'$ is generated in degrees $\le 3$, we see that \[\dim_{\bk} M'(\bZ^4) \le  \frac{35|\GL_4(\bF^3)|}{|\GL_3(\bF^3) \times \GL_1(\bF^3)| } \le 37800.\] Finally, we conclude that $\dim_{\bk} M(\bZ^4) \le 189540 + 37800 = 227340$. This completes the proof.
\end{proof}

\subsection{Proof of \autoref{prop:bound-p-3}}

\label{sec:proof-of-technical-proposition}

 Using the monoidal equivalence $\iota^{\ast} \colon \Mod_{\VB_{\bF_3}} \to \Mod_{\VB_{\bZ}/\Gamma(3)}$ as in \autoref{exs:gamma-p}, we now identify $\Mod_{\VB_{\bZ}/\Gamma(3)}$ with $\Mod_{\VB_{\bF_3}}$. \autoref{prop:bound-p-3} is immediate from \autoref{thm:shably-main}, the following proposition, and the fact that $\dim_\bk \St_n(\mathbb F_q) = q^{\binom{n}{2}}$.

\begin{proposition}
	\label{prop:LS-case3}
	The monoid $\bA_{\Gamma(3)}$ is naturally isomorphic to the apartment monoid in $\Mod_{\VB_{\bF_3}}$, and under this isomorphism, $\St_{\Gamma(3)}$ is the Steinberg monoid in $\Mod_{\VB_{\bF_3}}$.
\end{proposition}
\begin{proof}
The first assertion follows immediately from \autoref{prop:tensor-commute} and the fact that the apartment monoid is an exterior algebra. 

Let us denote the apartment and the Steinberg monoids in $\Mod_{\VB_{\bF_3}}$ by $\ol{\bA}$ and $\ol{\St}$ respectively.  By \autoref{presentationZ} and \autoref{fieldPresentation} (also see the proof of 	\autoref{prop:tor-0-1}), we have an exact sequence of $\bA$-modules \begin{align}\label{seqZ} \bA \otimes M \to \bA \to \St \to 0  \end{align} and an exact sequence of $\ol{\bA}$-modules \begin{align}\label{seqF3} \ol{\bA} \otimes \ol{M} \to \ol{\bA} \to \ol{\St} \to 0. \end{align} Here $M$ and $\ol{M}$ are supported only in degree 2 and are given by  \[M(\bZ^2) = \langle[v_1, v_2] - [v_0, v_2] + [v_0,v_1] \mid v_0 =  v_1 + v_2\rangle \subset \bA(\bZ^2) \] and  \[\ol{M}(\bF_3^2) = \langle[v_1, v_2] - [v_0, v_2] + [v_0,v_1] \mid v_0 = v_1 + v_2  \rangle \subset \ol{\bA}(\bF_3^2).\]   We note that the units in $\Z$ and $\bF_3$ are both $\{1,-1\}$.  Applying the monoidal functor $\rH_0(\Gamma(3); -)$ to \autoref{seqZ}, we obtain an exact sequence: \begin{align}\label{seqQuotient} \bA_{\Gamma(3)} \otimes \rH_0(\Gamma(3); M) \to \bA_{\Gamma(3)} \to \St_{\Gamma(3)} \to 0\end{align}  which we will show is isomorphic to \autoref{seqF3}. Under the isomorphism $\bA_{\Gamma(3)} \to \ol{\bA}$ as in the first assertion, we note that both $\St_{\Gamma(3)}$ and $\ol{\St}$ are quotients of $\ol{\bA}$ by an ideal generated in degree 2. To show that the two ideals are the same, all we need is to check that the image of $\rH_0(\Gamma_2(3); M(\bZ^2)) \to \rH_0(\Gamma_2(3); \bA(\bZ^2))$ is equal to the image of $\ol{M}(\bF_3^2) \to \ol{\bA}(\bF_3^2)$ under the isomorphism $\rH_0(\Gamma_2(3); \bA(\bZ^2)) \to \ol{\bA}(\bF_3^2)$. This last isomorphism is just reduction mod $p$ and is given by \[[v_1, v_2] \mapsto [\ol{v}_1, \ol{v}_2]. \] The statement about images being equal now follows immediately from the explicit descriptions of $M$ and $\ol{M}$ (see \cite[Lemma~5.2]{LS} for a similar argument).
\end{proof}

\begin{remark} Lee--Szczarba \cite[Theorem 1.2]{LS} proved that $\oSt_n(\bZ)_{\Gamma_n(3)} \cong \oSt_n(\bF_3)$ for $n \geq 3$. \autoref{prop:LS-case3} shows that this is in fact true for all $n$.
\end{remark}

\section{A stability conjecture}
\label{sec:stability-conjectures}

 In \cite{CFPconj}, Church, Farb, and Putman conjectured that for all $i$, the codimension $i$ rational cohomology of mapping class groups, automorphism groups of free groups, and $\SL_n(\Z)$ stabilize. Here codimension $i$ means $i$ below the virtual cohomological dimension. These groups are rational duality groups so this stability conjecture is equivalent to homological stability with coefficients in the dualizing modules tensor $\Q$. In particular, they conjectured that $\rH_i(\SL_n(\Z);\oSt_n(\Z) \otimes \Q)$ does not depend on $n$ for $n \gg i$.  In the case of $\rH_i(\SL_n(\Z);\oSt_n(\Z) \otimes \Q)$, if one could show that these groups stabilized, it would also imply that the stable homology groups are zero \cite[\S2]{CFPconj}. 
 
Since it is known that the cohomology groups of congruence subgroups in the virtual cohomological dimension are nonzero and in fact grow with $n$ \cite{LS,Par}, Church, Farb, and Putman did not conjecture that the cohomology of congruence subgroups vanishes or stabilizes  in high dimensions. Nevertheless, they said regarding the codimension $i$ cohomology of congruence subgroups that they ``do expect that the stability conjectured in Conjecture 1 should persist in some form'' \cite[Page 7]{CFPconj}. Given that the untwisted homology of congruence subgroups exhibits representation stability, it is reasonable to conjecture that the stability pattern exhibited by the homology  of congruence subgroups with coefficients in Steinberg modules should also be a form of representation stability. We propose the following conjecture as a way of making Church--Farb--Putman's statement more precise.

\begin{conjecture}
	\label{conj:stability-conjecture} Let  $p$ be a prime. For each $i , j \ge 0$, the $\VB_{\bZ}/\Gamma(p)$-module \[\Tor^{\bA_{\Gamma(p)}}_{i}(\bk, \rH_j(\Gamma(p); \St))\] is supported in finitely many degrees.
\end{conjecture}

%Before we unpack what this conjecture means in the some special cases, we note that most representation stability results currently in the literature can be rephrased in terms of vanishing of certain $\Tor$  groups. 

For simplicity, we concentrate on the case when $p$ is an odd prime so that the units of $\Z$ will inject into the units of $\bF_p$ and so that $\Gamma_n(p)$ will be torsion free and hence exhibit integral Borel--Serre duality. As explained in \autoref{exs:gamma-p},  $\VB_{\bZ}/\Gamma(p)$ is monoidally equivalent to an oriented version $\VB_{\bF_p}^{\pm}$ of $\VB_{\bF_p}$. Moreover, it induces a monoidal equivalence between $\Mod_{\VB_{\bZ}/\Gamma(p)}$ and $\Mod_{\VB_{\bF_p}^{\pm}}$. We now define an oriented version of the apartment monoid in $\Mod_{\VB_{\bF_p}^{\pm}}$.  Recall that an orientation as in \autoref{exs:gamma-p} is a generator of top exterior power of $X$ up to multiplication by $\pm 1$. The {\bf oriented apartment monoid} denoted $\ol{\bA}^{\pm}$ is given on a vector space $X$ of dimension $n$ and orientation $\ro$  as follows. It is generated as a $\bk$-module by symbols $[v_1, \ldots, v_n]$, one for each basis $v_1, \ldots, v_n$ of $X$ satisfying $\lw_{i=1}^n v_i = \pm \ro$, subject to the following relations: \begin{enumerate}
		\item $[v_1, \ldots, v_n] =\sgn(\sigma)[v_{\sigma(1)}, \ldots, v_{\sigma(n)}] $ for $\sigma$ a permutation.
		\item $[- v_1, v_2, \ldots, v_n] = [v_1, \ldots, v_n]$. 
	\end{enumerate} In other words,  $\ol{\bA}^{\pm}$ is the exterior algebra $\lw(\triv_1)$ in $\Mod_{\VB_{\bF_p}^{\pm}}$ where $\triv_1$ is supported only in degree $1$, and $\triv_1(1)$ is the trivial representation of $\GL^{\pm}_1(\bF_p)$. By \autoref{prop:tensor-commute}, we know that the monoidal equivalence $ \Mod_{\VB_{\bZ}/\Gamma(p)} \to \Mod_{\VB_{\bF_p}^{\pm}}$ takes $\bA_{\Gamma(p)}$  to the oriented apartment monoid $\ol{\bA}^{\pm}$. By Borel--Serre duality, our conjecture for $i =0$ is the statement that the $\GL_n^\pm(\bF_p)$-equivariant map \[ \Ind_{\GL_{n-1}^\pm(\bF_p) \times \GL_{1}^\pm(\bF_p)}^{\GL_n^\pm(\bF_p)} \rH^{\binom{n-1}{2} -j}(\Gamma_{n-1}(p) ; \bk)  \to \rH^{\binom{n}{2} -j}(\Gamma_n(p);\bk ) \]  is surjective for $n$ sufficiently large. In other words,   the $i=0$ case of \autoref{conj:stability-conjecture}  is the statement that the sequences \[ \left\{ \rH^{\binom{n}{2} -j}\left(\Gamma_{n}(p) ; \bk \right) \right\}_{\geq n}\] have finite generation degrees. Since $\GL^\pm_n(\bF_3)=\GL_n(\bF_3)$, \autoref{thm:main} establishes this conjecture for $i=0$, $j=1$, and $p=3$.

Since exterior algebras are Koszul, there is a Koszul resolution \[ \ol{\bA}^{\pm} \otimes \Sym^{\ast}(\triv_1) \to \bk \to 0.  \] This resolution can be used to calculate $\Tor_{i}^{\ol{\bA}^{\pm}}(\bk, M)$ for any $\ol{\bA}^{\pm}$-module $M$. In particular, if $\deg \Tor_{i}^{\ol{\bA}^{\pm}}(\bk, M) \le d$ for $i = 0, 1$ then the following sequence is exact in degrees $> d$:  \[ M \otimes \Sym^{2}(\triv_1) \to  M \otimes \Sym^{1}(\triv_1) \to M \to 0.  \] Concretely, if we think of $M$ as a sequence $\{M_n \}_{\ge 0}$ of representations of $\GL_n^{\pm}(\bF_p)$, then we have \[M_n = \coker\left(\Ind_{\GL_{n-2}^\pm(\bF_p) \times (\Z/2 \wr S_2)}^{\GL_n^\pm(\bF_p)}  M_{n-2} \otimes_{\bk} \bk \to \Ind_{\GL_{n-1}^\pm(\bF_p) \times \GL_1^\pm(\bF_p)}^{\GL_n^\pm(\bF_p)} M_{n-1} \otimes_{\bk} \bk \right)  \] for $n > d$. In particular, in a stable range, the representation of $M_n$ is determined by the representations $M_{n-1}$ and $M_{n-2}$ along with the stabilization map $M_{n-1} \m M_{n-2}$. This phenomenon is often called {\bf central stability} \cite{PU, Pa2} or {\bf finite presentation degree} \cite{CE}. We shall refer to $d$ as the {\bf central stability degree} of the sequence $\{M_n \}_{n\ge 0}$. We see that the conjecture for $i=0$ and $1$ is the statement that the following sequence has finite central stability degree: \[\left\{ \rH^{\binom{n}{2} -j}\left(\Gamma_{n}(p) ; \bk \right) \right\}_{n\geq 0}.\] We prove this for $j=0$ and all $p$ in the proposition below. We need a definition first.  Given a prime field $R = \Fp$, we define the {\bf oriented Steinberg monoid}  $\ol{\St}^{\pm} = \St_{\Gamma(p)}$. Using Bykovskii's presentation (see \autoref{presentationZ}) the following presentation of $\ol\St^{\pm}$ is immediate.  On a vector space $X$ of dimension $n$ and orientation $\ro$  as follows, it is generated by as a $\bk$-module by symbols $[v_1, \ldots, v_n]$, one for each basis $v_1, \ldots, v_n$ of $X$ satisfying $\lw_{i=1}^n v_i = \pm \ro$, subject to the following relations: \begin{enumerate}
	\item $[v_1, \ldots, v_n] =\sgn(\sigma)[v_{\sigma(1)}, \ldots, v_{\sigma(n)}] $ for $\sigma$ a permutation.
	\item $[- v_1, v_2, \ldots, v_n] = [v_1, \ldots, v_n]$. 
	\item $[v_1, v_2, \ldots, v_n] - [v_0, v_2, \ldots, v_n] + [v_0, v_1, \ldots, v_n] = 0$ where $v_0= v_1 +  v_2$. 
\end{enumerate} Note that in $(b)$, we are only allowed to scale by $-1$, not arbitrary units. This is what differentiates $\ol{\St}^{\pm}_n(\bF_p)$ from $\St_n(\bF_p)$ for $p>3$. It is clear from this description that $ \deg \Tor^{\ol{\bA}^{\pm}}_{i}(\bk, \ol{\St}^{\pm}) \le 2i $ for $i =0,1$. Thus the following sequence has central stability degree $\le 2$: \[\left\{ \ol{\St}_n^{\pm}(\bF_p) \right \}_{n\geq 0} =\left\{ \rH^{\vcd}(\Gamma_n(p);\bk)\right \}_{n\geq 0}.\]

\begin{proposition} \label{otherPrimes}
	Let  $p$ be an odd prime.  Then  $ \deg \Tor^{\bA_{\Gamma(p)}}_{i}(\bk, \rH_0(\Gamma(p); \St)) \le 2i $ for $i =0,1$. In other words, the sequence $\left \{ \rH^{\binom{n}{2}}(\Gamma_n(p) ; \bk ) \right \}_{n \geq 0}$ has generation degree $0$ and central stability degree $\leq 2$. 
\end{proposition}
\begin{proof}
	An argument similar to the one in \autoref{prop:LS-case3}, shows that the monoidal equivalence $\Mod_{\VB_{\bZ}/\Gamma(p)} \to \Mod_{\VB_{\bF_p}^{\pm}}$ takes $\bA_{\Gamma(p)}$  to the oriented apartment monoid $\ol{\bA}^{\pm}$, and the $\bA_{\Gamma(p)}$-module $\St_{\Gamma(p)}$ to the $\ol{\bA}^{\pm}$-module $\ol{\St}^{\pm}$.  The result now follows from the previous paragraph.
\end{proof}

 When $p=3$, we know by Lee--Szczarba's result \cite[Theorem~1.2]{LS} that $\St_{\Gamma(p)} \cong \ol{\St}$  (see  \autoref{prop:bound-p-3} and \autoref{prop:LS-case3}). We see that \[\Tor^{\bA_{\Gamma(p)}}_{i}(\bk, \rH_0(\Gamma(p); \St)) \cong \Tor_i^{\ol{\bA}}(\bk, \ol{\St}).  \] Thus, \autoref{thm:shably-main} implies \autoref{conj:stability-conjecture} for $p=3$, $j=0$, and all $i$. 

 %\begin{remark}
%If one could prove \autoref{conj:stability-conjecture} for $i=0$ and $j=2$, the arguments of this paper would also establish a version of \autoref{thm:main} for all primes. By the work of Paraschivecu \cite{Par},
% \[ \rH^{\binom{n}{2} }(\Gamma_n(p)) \ncong \oSt_n(\bF_p) \text{ for } p>3. \] Thus, \autoref{thm:shably-main} is not directly relevant for primes larger than $3$. 
 
%There is no known conceptual description of the top homology group of congruence subgroups for $p>3$ and $n$ large nor are the ranks of these groups known.  Lee-Szczarba \cite[Page 28]{LS} conjectured that \[ \rH^{\binom{n}{2} }(\Gamma_n(p)) \cong \widetilde \rH_{n-2}(\cT_n(\Q)/\Gamma_n(p)).\] However, in \cite{MPP}, it is shown that this conjecture is wrong for $p>5$. 
%\end{remark}

\section{Homological vanishing for the Steinberg module}
\label{sec:homological-vanishing}

%Fix a field $K$, and set $\VB = \VB_{K}$. Let $\St$ be the Steinberg monoid in $\Mod_{\VB}$. 

In \cite{APS}, Ash--Putman--Sam proved that $\rH_i(\GL_n(K);\oSt_n(K))$ vanishes for $K$ a field and $n$ sufficiently large compared with $i$. Similarly, Church--Putman \cite{CP} proved that $\rH_1(\GL_n(\Z);\oSt_n(\Z) \otimes \bQ)$ vanishes for all $n \ge 0$ and that $\rH_1(\SL_n(\Z);\oSt_n(\Z) \otimes \bQ)$ vanishes for $n \ge 3$. In this section, we give a new proof of the result of Ash--Putman--Sam and give integral refinements of these results of Church--Putman.

\subsection{Homological vanishing for the Steinberg module of a field}

Ash--Putman--Sam proved the following. 

\begin{theorem}[Ash--Putman--Sam \cite{APS}]
	\label{thm:homological-vanishing} Let $K$ be a field. Then 	$\rH_i(\GL_n(K); \oSt_n(K) \otimes_{\bZ}\bk) = 0$ for all $n \ge 2 i + 2$. 
\end{theorem}

We use Koszulness of the Steinberg monoid (\autoref{thm:Koszulness-of-Steinberg}) to give a new proof of the theorem above. Unlike the original proof, our proof does not make use of high connectivity of the complex of partial bases. We will deduce homological vanishing for the Steinberg module using the following general criterion for homological vanishing for Koszul monoids which may be of independent interest. Note this is equivalent to \autoref{intro:general-vanishing}.

\begin{theorem}
	\label{thm:general-vanishing}
	Let $A$ be a (skew) commutative Koszul monoid in $\Mod_{\VB}$. Assume that $\bk$ is a field such that the following holds: \begin{enumerate}
		\item $\rH_0(\GL_2; A_2) = 0$.
		\item The product map $\rH_0(\GL_1; A_1) \otimes \rH_1(\GL_2; A_2) \to \rH_1(\GL_3; A_3)$ is surjective.
	\end{enumerate} Then we have that: \begin{enumerate}[(a')]
	\item The product map $\rH_0(\GL_1; A_1) \otimes \rH_i(\GL_{n-1}; A_{n-1}) \to \rH_i(\GL_n; A_n)$ is surjective for $n \ge 2i+1$.
	\item $\rH_i(\GL_n; A_n) = 0$ in degrees $n \ge 2i+2$.
\end{enumerate}
\end{theorem}

We shall refer to (a) and (b) in the theorem above as hypotheses or initial conditions. 

\begin{remark} \label{signRep}
	As noted earlier, one can generalize our homological vanishing criterion for Koszul monoids replacing general linear groups with certain other families of groups. 	Our hypotheses $(a)$ and $(b)$ are likely necessary. In particular, they are necessary if one replaces general linear groups by symmetric groups and if one takes $A_n$ to be the sign representations. A result contributed to Pierre Vogel \cite[Proposition~B]{MR483534} gives explicit calculations that show that the homology does not vanish in a slope $\frac{1}{2}$ range.  We note that without Hypothesis $(b)$, the homology still vanishes in a slope $\frac{1}{3}$ range. This can be proven by adapting the proof of \autoref{thm:general-vanishing}. In particular, Hypothesis $(a)$ implies that the product map $\rH_0(\GL_1; A_1) \otimes \rH_1(\GL_3; A_3) \to \rH_1(\GL_4; A_4)$ is surjective which can be used in place of Hypothesis $(b)$.

%	The exterior algebra $A$ in $\prod_{n \ge 0} \Mod_{\bZ[S_n]}$ is Koszul and we have $A_n = \sgn_n$. For $\bk$ a field of characteristic $2$, Hypothesis $(a)$ fails. We have $\rH_i(S_n; \sgn_n) \cong \rH_i(S_n;\bk)$ which does not vanish and so conclusion $(b')$ does not hold.
	%Now consider the case $\bk$ is a field of characteristic $3$. It follows from the work of Vogel (see \cite[Proposition~B]{MR483534}) that the Hypothesis (b) in the theorem above (with $\GL_n$ replaced by $S_n$) is false. Moreover, it is known that $\rH_i(S_n; \sgn_n)$ only vanishes for $n \geq 3i+2$ \cite{MR483534}. This is in some sense all that can go wrong.
	%If $\bk$ is a field of characteristic greater than $3$, then Hypotheses $(a)$ and $(b)$ hold \cite[Proposition~B]{MR483534}. This gives a new proof of vanishing for $\rH_i(S_n; \sgn_n)$ in the range $n \ge 2i + 2$ after inverting $6$. These results can be used to show homological stability for alternating groups. 

	%	If $\bk$ is a field of characteristic 3. Moreover, it is clear that  Hypothesis (a) above (with $\GL_n$ replaced by $S_n$) is false in this case if and only if $\bk$ is a field of characteristic 2. Thus our theorem states that $n \ge 2i + 2$ is the vanishing range for the homology of sign representation away from characteristics $2$ and $3$. 
\end{remark}

We use  the setup in \autoref{sec:spectral-sequence} with $\cG = \GL$ to prove the theorem. In this case,  $\Mod_{\VB/\GL}$ is equivalent to $\Mod_{\bk}^{\ge 0}$ -- the category of non-negatively graded $\bk$-modules.

\begin{proposition}
	\label{prop:bar-spectral-sequence}
There is a spectral sequence of non-negatively graded $\bk$-modules with \[E^1_{ab} =\rH_a(\GL; \ol{\cB}_{b}(A))   \implies \rL_{a+b}(\rH_0(\GL; -) \circ (\bk\otimes_{A} -) )(\bk). \] For $A$ Koszul, we have that $E^{\infty}_{ab}$ is supported in degrees $\le (a+b)$.
\end{proposition}

\begin{proof}

We have the following equality of the total derived functors \[\rL(\rH_0(\GL; -) \circ (\bk\otimes_{A} -) )(\bk) = (\rL\rH_0(\GL; -) \circ (\bk\otimes_{A}^{\rL} -))(\bk) = \rL\rH_0(\GL; \bk\otimes_{A}^{\rL} \bk)  \] where the first equality holds because the hypotheses for the Grothendieck spectral sequence are satisfied. This involves noting that the functor $\bk \otimes_A - \colon \Mod_A \to \Mod_{\VB}$ preserves projectives -- any projective $A$-module is of the form $A \otimes V$ where $V$ is a projective $\VB$-module. Since the bar resolution for $A$ is acyclic with respect to the functor $\bk \otimes_{A} -$, the total derived functor $\bk\otimes_{A}^{\rL} \bk$ is represented by the complex $\ol{\cB}_{\ast}(A)$ from \autoref{intro-koszul}. Thus there exists a spectral sequence  given by \[E^1_{ab} =\rH_a(\GL; \ol{\cB}_{b}(A))   \implies \rL_{a+b}(\rH_0(\GL; -) \circ (\bk\otimes_{A} -) )(\bk). \]

For the second statement, consider the Grothendieck spectral sequence \['E^2_{ab} =\rH_a(\GL; \Tor^{A}_b(\bk,\bk))   \implies \rL_{a+b}(\rH_0(\GL; -) \circ (\bk\otimes_{A} -) )(\bk)  \] as in the proof of \autoref{thm:spectral sequence}. Since $A$ is Koszul, we see that $'E^{2}_{ab}$ is supported in degrees $\le (a+b)$. This shows that $E^{\infty}_{ab}$ is supported in degrees $\le (a+b)$.
\end{proof}

% We also note here that the second page of this spectral sequence is not the Grothendieck spectral sequence mentioned above. Note, we used the same symbol $E^r_{ab}$ for this spectral sequence and the Grothendieck spectral sequence. 
%but we shall not use the Grothendieck spectral sequence anymore.)

%Richard Hepworth used a different technique to calculate the total derived functor  $\rL(\rH_0(\GL; -) \circ (\bk\otimes_{A} -) )(\bk)$, which gives nontrivial information. We now describe this technique.  We have the following equality of the total derived functors \[\rL(\rH_0(\GL; -) \circ (\bk\otimes_{A} -) )(\bk) = (\rL\rH_0(\GL; -) \circ (\bk\otimes_{A}^{\rL} -))(\bk) = \rL\rH_0(\GL; \bk\otimes_{A}^{\rL} \bk)  \] where the first equality holds because the hypotheses for the Grothendieck spectral sequence are satisfied (involved functors preserve projectives; see \autoref{sec:spectral-sequence}). Since the bar resolution for $A$ is acyclic with respect to the functor $\bk \otimes_{A} -$, the total derived functor $\bk\otimes_{A}^{\rL} \bk$ is represented by the complex $\ol{\cB}^n_{\ast}(A)$ from \autoref{intro-koszul}. Thus there exists a spectral sequence  given by \[E^1_{ab} =\rH_a(\GL; \ol{\cB}^n_{b}(A))   \implies \rL_{a+b}(\rH_0(\GL; -) \circ (\bk\otimes_{A} -) )(\bk). \] This is the spectral sequence that Hepworth used for the Koszul algebra built out of trivial representations; see \cite{HepEdge}.

\begin{lemma}
	Assume that $\bk$ is a field. 
\begin{enumerate}
	\item 	If \autoref{thm:general-vanishing} holds for $i \le d$, then for any  $a \le d$ and $b >0$, we have \[\deg \rH_a(\GL; \ol{\cB}_{b}(A)) \le 2 a +b.\] 
	
	\item 	If \autoref{thm:general-vanishing} holds for $i \le d$, then for any  $0 < a \le d$ and $b >0$, the differential \[ \rH_a(\GL; \ol{\cB}_{b+1}(A)) \to \rH_a(\GL; \ol{\cB}_{b}(A))\] is surjective in degree $2 a + b$.

	\item If $b =0$, then we have $\deg \rH_a(\GL; \ol{\cB}_{b}(A)) \le 0$.
\end{enumerate}	
\end{lemma}
\begin{proof}
Proof of Part (a): By definition, we have \[\rH_a(\GL; \ol{\cB}_{b}(A))  = \rH_a(\GL; A_{+}^{\otimes b}). \] If $\bk$ is a field, then by the K\"unneth formula, $\rH_a(\GL; A_{+}^{\otimes b})$ in degree $n$ is a direct sum of $\bk$-modules of the form \[ \bigotimes_{j = 1}^b \rH_{a_j}(\GL_{n_j}; A_{n_j})  \] where $\sum_{j=1}^b a_j = a$, $\sum_{j=1}^b n_j = n$ and $n_j >0$ for all $j$. If $n > 2 a + b$, then by the pigeonhole principle we have $n_j \ge 2 a_j +2$ for some $j$. Thus we have \[ \bigotimes_{j = 1}^b \rH_{a_j}(\GL_{n_j}; A_{n_j}) =0, \] completing the proof of Part (a).

Proof of Part (b): By the pigeonhole principle, $\rH_a(\GL; A_{+}^{\otimes b})$ in degree $ 2 a + b$ is a direct sum of $\bk$-modules of the form \[ \bigotimes_{j = 1}^b \rH_{a_j}(\GL_{n_j}; A_{n_j})  \] where $\sum_{j=1}^b a_j = a$, $n_j = 2 a_j + 1$ for all $j$. We say that such a direct summand is of {\bf type $j_0$} if $j_0$ is the largest such that $n_{j_0}  \ge 2$. This must exist as $a > 0$.  For a type $j_0$ direct summand as above, define  $a'_j$ and $n'_j$ for $1 \le j \le b+1$ by \begin{align*}
(a'_j, n'_j)  = \begin{cases}
(a_j, n_j) & \mbox{if } j < j_0,\\
(0, 1) & \mbox{if } j = j_0,\\
(a_{j_0}, n_{j_0} - 1) & \mbox{if } j = j_0 + 1,\\
(a_{j-1}, n_{j-1}) & \mbox{if } j > j_0 + 1.
\end{cases}
\end{align*} Then \[ \bigotimes_{j = 1}^{b+1} \rH_{a'_j}(\GL_{n'_j}; A_{n'_j})  \] is a direct summand of $\rH_a(\GL; \ol{\cB}_{b+1}(A)) $ and the differential is given by the alternating sum of multiplying two of the consecutive factors. Since the theorem holds for $a \le d$, we have the following: \begin{itemize}
	\item Multiplication of any of these consecutive factors besides $j_0$-th and $(j_0 + 1)$-th factors, or $(j_0 + 1)$-th and $(j_0 + 2)$-th factors is zero. 
	\item Multiplying $j_0$-th and $(j_0 + 1)$-th factors yields a surjective map \[ \bigotimes_{j = 1}^{b+1} \rH_{a'_j}(\GL_{n'_j}; A_{n'_j})  \to \bigotimes_{j = 1}^b \rH_{a_j}(\GL_{n_j}; A_{n_j}).\]
	\item Multiplying $(j_0 + 1)$-th and $(j_0 + 2)$-th factors takes  \[ \bigotimes_{j = 1}^{b+1} \rH_{a'_j}(\GL_{n'_j}; A_{n'_j}) \]  inside a direct summand of type $> j_0$.
\end{itemize} Let $\mathrm{Type}_{\ge j_0}$ denote the direct sum of summands of type $j$ for $j \ge j_0$. Our argument above shows, by an easy reverse induction on $j_0$,  that the image of the differential \[ \rH_a(\GL; \ol{\cB}_{b+1}(A)) \to \rH_a(\GL; \ol{\cB}_{b}(A)),\] in degree $2 a +b$,  contains $\mathrm{Type}_{\ge j_0}$ for any $j_0$. Thus the differential is surjective, completing the proof of Part (b).

Proof of Part (c): Part (c) is immediate from $A_{+}^{\otimes 0} = \bk$.
\end{proof}

We now use this lemma to prove our homological vanishing criterion for Koszul monoids.

\begin{proof}[Proof of \autoref{thm:general-vanishing}] 
The base case $i = -1$ is trivial. Assume now that $i \ge 0$ and that the theorem holds for $a < i$. Consider the spectral sequence \[E^1_{ab} =\rH_a(\GL; \ol{\cB}_{b}(A))   \implies \rL_{a+b}(\rH_0(\GL; -) \circ (\bk\otimes_{A} -) )(\bk) \] as in \autoref{prop:bar-spectral-sequence}, and assume that $n \ge 2i + 1$.   In the following two paragraphs, we analyze this spectral sequence of graded $\bk$-modules on the diagonals $a + b = i+1$ and $a+b = i+2$ in degree $n$.

Assume that $a + b = i+1$.  If $a < i$, then by induction and Part (a) of the previous lemma, we have \[\deg \rH_a(\GL; \ol{\cB}_{b}(A)) \le 2 a + b = 2(i+1) - b \le 2i < n.  \] Moreover if $a = i+1$, then by Part (c) of the previous lemma, we have $\deg \rH_a(\GL; \ol{\cB}_{b}(A)) < n$. In particular, the only nonzero entry in degree $n$ on the diagonal $a + b = i+1$ comes from $E^1_{i,1}$.

Assume that $a + b = i+2$. If  $a < i-1$, then by Part (a) of the previous lemma, we have \[\deg \rH_a(\GL; \ol{\cB}_{b}(A)) \le 2 a + b = 2(i+2) - b = 2i +4 - b < n.  \] If $a = i-1$ and $n > 2i+1$, then by Part (a) of the previous lemma, the differential \[ \rH_{i-1}(\GL; \ol{\cB}_{4}(A)) \to \rH_{i-1}(\GL; \ol{\cB}_{3}(A))\] is surjective in degree $n$ as the target vanishes. If $a = i-1 >0$ and $n = 2i+1$, then by Part (b) of the previous lemma, the differential \[ \rH_{i-1}(\GL; \ol{\cB}_{4}(A)) \to \rH_{i-1}(\GL; \ol{\cB}_{3}(A))\] is surjective in degree $n$. In the remaining case, $a = i-1 = 0$, $n = 2i + 1$, and the differential \[ \rH_{i-1}(\GL; \ol{\cB}_{4}(A)) \to \rH_{i-1}(\GL; \ol{\cB}_{3}(A))\] may not be surjective in degree $n$, and that is why we need an additional hypothesis, namely Hypothesis (b), in the case $i=1$.

We now use the paragraphs above to show that the differential \[ E^1_{i, 2} =  \rH_i(\GL; A_{+}^{\otimes 2}) \to \rH_i(\GL; A_{+}) = E^1_{i, 1} \] is surjective in degree $n$. The case $i = 0 $ and $n = 1$ follows from Hypothesis (a). In the remaining cases, we have $n \ge i + 2$. Since $E^{\infty}_{ab}$ is supported in degrees $\le (a+b)$, we see that $E^{\infty}_{ab}$ vanishes in degree $n$ on the line $a + b = i+1$. Since $n \ge i+2$, some differential at some page must kill the degree $n$ piece of the entry $E^1_{i,1}$. By the paragraph above, we have $E^2_{ab} = 0$ in degree $n$ for $a + b = i + 2$ for $a \le i-1$  except when $n = 2 i + 1 = i+2$. In this exceptional case, our claim follows from Hypothesis (b). Away from this exceptional case, we have $E^2_{ab} = 0$ in degree $n$ for $a + b = i + 2$ and $a \le i-1$. Since $E^{\infty}_{i,1} = 0$, we must have that \[ E^1_{i, 2} =  \rH_i(\GL; A_{+}^{\otimes 2}) \to \rH_i(\GL; A_{+}) = E^1_{i, 1} \] is surjective in degree $n$, proving our claim.

By induction, the theorem holds for $a < i$. Hence  the degree $n$ piece of $\rH_i(\GL; A_{+}^{\otimes 2})$ is given by \[\rH_i(\GL_{n-1}; A_{n-1}) \otimes \rH_0(\GL_{1}; A_{1}) \bigoplus   \rH_0(\GL_{1}; A_{1}) \otimes \rH_i(\GL_{n-1}; A_{n-1}). \] By the previous paragraph and the (skew) commutativity of $A$, we conclude that \[\rH_0(\GL_{1}; A_{1}) \otimes \rH_i(\GL_{n-1}; A_{n-1}) \to  \rH_i(\GL_{n}; A_{n}) \] is surjective for $n \ge 2 i+1$. This proves Part (a').

For Part (b'), note that \[  \rH_0(\GL_{1}; A_{1}) \otimes \rH_0(\GL_{1}; A_{1}) \otimes \rH_i(\GL_{n-2}; A_{n-2}) \to  \rH_i(\GL_{n}; A_{n}) \] is surjective by Part (a'). But, by associativity of multiplication, this map factors through \[ \rH_0(\GL_{2}; A_{2}) \otimes  \rH_i(\GL_{n-2}; A_{n-2}) \to  \rH_i(\GL_{n}; A_{n}), \] which vanishes by Hypothesis (a). This completes the proof.
\end{proof}

We now show that the initial conditions in \autoref{thm:general-vanishing} are satisfied when $A = \St$. Let  $T = \rT(\reg_1)$ be the algebra as in \autoref{prop:tensor-algebra-is-Koszul} for $R = K$. Then $\rH_0(\GL; T)$ is naturally isomorphic to the polynomial ring $\bk[t]$ which is a monoid in $\Mod_{\bk}^{\ge 0}$. Thus $\rH_0(\GL; -)$ induces a functor \[\rH^T_0(\GL; -) \colon \Mod_{T} \to \Mod_{\bk[t]}. \]  Note that there are  natural surjections $T \to \St \to \bk$ of monoids, and so $\St$ is a module over $T$. Note here that  \autoref{prop:map-of-Koszul-algebras} is not applicable as $T$ is not (skew) commutative. Now we apply the Grothendieck spectral sequence from the proof of \autoref{thm:spectral sequence} to the $T$-module $\St$ to obtain: \[E^2_{ab} = \Tor^{\bk[t]}_a(\bk, \rH^T_b(\GL; \St))  \implies \rL_{a+b}((\bk\otimes_{\bk[t]} -) \circ \rH^T_0(\GL; -) )(\St). \]  A version of \autoref{lem:char0-is-good} for $T$, implies that $\rH^{T}_{q}(\GL; -)$ is isomorphic to $\rH_{q}(\GL; -)$ as a graded $\bk$-module and so we shall now drop the superscript $T$ from our notation. Note that $\rH_{q}(\GL; -)$ has the structure of a $\bk[t]$-module. The spectral sequence above simplifies to \[E^2_{ab} = \Tor^{\bk[t]}_a(\bk, \rH_b(\GL; \St))  \implies \rL_{a+b}((\bk\otimes_{\bk[t]} -) \circ \rH_0(\GL; -) )(\St). \] 

\begin{lemma}
	\label{lem:L1-reduction}
If $\rL_{1}((\bk\otimes_{\bk[t]} -) \circ \rH_0(\GL; -) )(\St)$ vanishes in degree $d$, then the product map \[\rH_0(\GL_1; \St_1) \otimes \rH_1(\GL_{d-1}; \St_{d-1}) \to \rH_1(\GL_d; \St_d)\] is surjective.
\end{lemma}
\begin{proof}
	Since $\Tor^{\bk[t]}_i = 0$ for $i > 1$, we have $E^2_{ab} = E^{\infty}_{ab}$. Thus, if $\rL_{1}((\bk\otimes_{\bk[t]} -) \circ \rH_0(\GL; -) )(\St)$ vanishes in degree $d$, then $E^2_{0,1}$ vanishes in degree $d$. Since $\bk[t] = \rH_0(\GL; T)$, this is equivalent to the surjectivity of the product map \[\rH_0(\GL_1; T_1) \otimes \rH_1(\GL_{d-1}; \St_{d-1}) \to \rH_1(\GL_d; \St_d). \] The assertion now follows from \autoref{prop:factoring-through} Part (b) applied to $A = T$ and $B = \St$.
\end{proof}

To calculate $\rL_{1}((\bk\otimes_{\bk[t]} -) \circ \rH_0(\GL; -) )(\St)$, we need a resolution of $\St$ by projective $T$-modules. The Lee--Szczarba resolution $\LS_{\ast}$ is a such a resolution. 

Let $X$ be an object in $\VB = \VB_K$ of rank $n$. Let $\LS_i$ be the  $\VB$-module defined as follows.  The module $\LS_i(X)$ is a $\bk$-module generated by symbols of the form $[v_1, \ldots, v_{n+i}]$ where the $v_j$ are distinct nonzero elements of $X$ and where we impose the following relation: \[ [v_1, \ldots, v_{n+i}] = 0\text{ if  }v_1, \ldots, v_n \text{ do not span }X. \] There is an equivariant map \[T(X_1) \otimes \LS_i(X_2)  \to \LS_i(X) \] given by \[ [u_1, \ldots, u_{k}] [v_1, \ldots, v_{n-k + i}] = [u_1, \ldots, u_k, v_1, \ldots, v_{n-k+i}]. \] This gives $\LS_i$ the structure of a $T$-module. It is easy to see that $\LS_i$ is a projective $T$-module. There is a natural differential $\LS_{i} \to \LS_{i-1}$ given by the alternating sum of forgetting vectors. By \cite[Theorem~3.1]{LS}, $\LS_{\ast} \to \St$ is a $T$-projective resolution. In particular, we have \[\rL_{i}((\bk\otimes_{\bk[t]} -) \circ \rH_0(\GL; -) )(\St) = \rH_i(\bk\otimes_{\bk[t]} \rH_0(\GL; \LS_{\ast})). \] Since $\deg \bk\otimes_{\bk[t]} \rH_0(\GL; \LS_{0}) = 0$, any element of $\bk\otimes_{\bk[t]} \rH_0(\GL; \LS_1)$ in positive degrees is a cycle. Thus to show that $\rL_{1}((\bk\otimes_{\bk[t]} -) \circ \rH_0(\GL; -) )(\St)$ vanishes in degree $3$, it suffices to show that the following map is surjective in degree 3: \[ d_2 \colon \bk\otimes_{\bk[t]} \rH_0(\GL; \LS_2) \to \bk\otimes_{\bk[t]} \rH_0(\GL; \LS_1). \] Recall that a minimal linearly dependent subset of a vector space is called a {\bf circuit}. 

\begin{lemma} 
	\label{lem:L1-vanishing}
	$\rL_{1}((\bk\otimes_{\bk[t]} -) \circ \rH_0(\GL; -) )(\St)$ vanishes in degree $3$.
\end{lemma}
\begin{proof} 
		Let $v_1, v_2, v_3, v_4$ be nonzero vectors in $K^3$. By the previous paragraph, it suffices to show that the image of $[v_1, v_2, v_3, v_4]$ in $\coker d_2$ vanishes. We do this case-by-case as follows.

\begin{enumerate}
	\item[{\bf Case (a)}.] $\{v_2, v_3, v_4 \}$ is not a basis of $K^3$. 
	
	We can assume without loss of generality that $\{v_1, v_2, v_3, v_4\}$ spans $K^3$. The hypothesis implies that  $v_1$ does not lie in the span of $\{v_2, v_3, v_4 \}$. Thus we have \[[v_1, v_2, v_3, v_4]  = [v_1][v_2, v_3, v_4] \in T_{+} \LS_1,\] and so its image  in $\rH_0(\GL; \LS_1)$ is contained in $\bk[t]_{+}\rH_0(\GL; \LS_1)$. The assertion follows from this.
	
	\item[{\bf Case (b)}.] $\{v_1, v_2, v_3, v_4\}$ is a circuit. 
	
	Note that we have \begin{align*}
	d_2([v_1,v_2,v_3 + v_4, v_3, v_4]) &= [v_2,v_3 + v_4, v_3, v_4] - [v_1,v_3 + v_4, v_3, v_4] + [v_1,v_2, v_3, v_4]\\ &-[v_1,v_2,v_3 + v_4, v_4] + [v_1,v_2,v_3 + v_4, v_3].
	\end{align*} By Case (a), the first two terms vanish in $\coker d_2$. Let $g \in \GL_3(K)$ be the linear involution which fixes $v_1$ and $v_2$ and takes $v_3$ to $v_4$. Then $g$ interchanges the last two terms in the expression above, and so the last two terms cancel each other out in $\coker d_2$. Thus the middle term vanishes in $\coker d_2$, completing the proof in Case (b).
	
	\item[{\bf Case (c)}.] $\{v_2, v_3, v_4 \}$ is a basis of $K^3$ and $v_1$ is a constant multiple of either $v_2$, $v_3$, or $v_4$.
		
	Let $u = v_2 + v_3 + v_4$. Then we have \begin{align*}
	d_2([u,v_1,v_2,v_3, v_4]) = [v_1,v_2,v_3,v_4] - [u, v_2,v_3,v_4] + [u, v_1, v_3,v_4] -[u, v_1, v_2, v_4] + [u,v_1,v_2,v_3].
	\end{align*} Each of the last four terms fall under Case (a) or Case (b), and so the images of the last four terms vanish in $\coker d_2$. Thus the image of the first term vanishes in $\coker d_2$, completing the proof in Case (c).
	
	\item[{\bf Case (d)}.] $\{v_2, v_3, v_4 \}$ is a basis of $K^3$ and $v_1 = a_2 v_2  + a_3 v_3 + a_4 v_4$ such that exactly one of $a_2, a_3, a_4$ is 0.
	
	The hypothesis implies that there exists a unique $i \in \{2,3,4\}$ such that $a_i = 0$. Set $u = v_i$. Then we have \begin{align*}
	d_2([u,v_1,v_2,v_3, v_4]) = [v_1,v_2,v_3,v_4] - [u, v_2,v_3,v_4] + [u, v_1, v_3,v_4] -[u, v_1, v_2, v_4] + [u,v_1,v_2,v_3].
	\end{align*}  For each of the last four terms, if $v_i$ is present in it then the term falls under Case (c), and if $v_i$ is not present in it then the term falls under Case (a). Thus the last four terms vanish in $\coker d_2$. This shows that the first term vanish in $\coker d_2$, completing the proof in Case (c).
\end{enumerate}
	
The four cases above are exhaustive, and so our proof is complete.
\end{proof}

\begin{proof}[Proof of \autoref{thm:homological-vanishing}] 
It suffices to prove the theorem when $\bk$ is a field as the general case follows from the field case via the universal coefficient theorem. By \autoref{thm:Koszulness-of-Steinberg}, we know that $\St$ is Koszul. Moreover, we have $\St_n(K) = \oSt_n(K) \otimes_{\bZ} \bk$. Thus it suffices to verify Hypotheses (a) and (b) in \autoref{thm:general-vanishing}. Hypothesis (a) follows from \cite[Lemma~2.6]{APS}, and Hypothesis (b) follows from \autoref{lem:L1-reduction} and \autoref{lem:L1-vanishing}.
\end{proof}

\subsection{Homological vanishing for the Steinberg module of the integers}

We now prove integral refinements of Church--Putman's homological vanishing theorem (\autoref{thm:CP-vanishing-char-0}) for the Steinberg module of the integers. 
  
\begin{theorem} 
	Let $\bk$ be an arbitrary commutative ring. For $n \ge 6$, we have that: \begin{align*}
	\rH_1(\GL_n(\bZ); \oSt_n(\bZ) \otimes_{\bZ}\bk) &= 0, \\
	\rH_1(\SL_n(\bZ); \oSt_n(\bZ) \otimes_{\bZ}\bk) &= 0.
	\end{align*} 
\end{theorem}
\begin{proof} It suffices to prove the assertion for $\SL$. The Lyndon--Hochschild--Serre spectral sequence corresponding to the short exact sequence \[1 \to \Gamma_n(3) \to \SL_n(\bZ) \to \SL_n(\bF_3) \to 1  \] is given by \[\rH_p(\SL_n(\bF_3); \rH_q(\Gamma_n(3); \oSt_n(\bZ)\otimes_{\bZ} \bk )) \implies \rH_{p+q}(\SL_n(\bZ); \oSt_n(\bZ)\otimes_{\bZ} \bk )   \]  Thus it suffices to show that \begin{align*}
\rH_1(\SL_n(\bF_3); \rH_0(\Gamma_n(3); \oSt_n(\bZ)\otimes_{\bZ} \bk )), \\
\rH_0(\SL_n(\bF_3); \rH_1(\Gamma_n(3); \oSt_n(\bZ)\otimes_{\bZ} \bk ))
	\end{align*} vanish for $n \ge 6$. By \cite[Theorem~1.2]{LS}, we know that $\rH_0(\Gamma_n(3); \oSt_n(\bZ)\otimes_{\bZ} \bk ) = \oSt_n(\bF_3) \otimes_{\bZ} \bk$. Thus  $\rH_1(\SL_n(\bF_3); \rH_0(\Gamma_n(3); \oSt_n(\bZ)\otimes_{\bZ} \bk )) = 0 $ for $n \ge 4$ by the homological vanishing result of Ash--Putman--Sam (\cite[Theorem~1.1]{APS}).

 Let $\ol{\St}$ be the Steinberg monoid  in $\Mod_{\VB_{\bF_3}}$.  By \autoref{prop:LS-case3} and \autoref{cor:fancy-St},   $\rH_1(\Gamma(3), \St)$ is an  $\ol{\St}$-module generated in degrees $\le 4$. In particular, if $n \ge 4$ then we have a surjection \[
 \Ind_{\GL_{n-4}(\bF_3) \times \GL_4(\bF_3)}^{\GL_n(\bF_3)} \St_{n-4}(\bF_3) \otimes_{\bk}  \rH_1(\Gamma_4(3); \St_4(\bZ)) \to \rH_1(\Gamma_n(3); \St_n(\bZ)). \] Now if $n \ge 6$, then $n-4 \ge 2$, and so by \cite[Lemma~2.6]{APS} we have $\rH_{0}(\SL_{n-4}(\bF_3) ; \St_{n-4}(\bF_3)) = 0$.  This shows that $\rH_0(\SL_n(\bF_3); \rH_1(\Gamma_n(3); \St_n(\bZ))) = 0$ for $n \ge 6$. This completes the proof after noting that $\oSt_n(\bZ)\otimes_{\bZ} \bk = \St_n(\bZ)$.
\end{proof}

%\rohit{Shall we add a remark about the minor mistake in \cite[Lemma~3.1]{APS} or shall we just not bother about that?}

\bibliographystyle{amsalpha}
\bibliography{CodimOne}

\newcommand{\etalchar}[1]{$^{#1}$}
\def\cprime{$'$}
\providecommand{\bysame}{\leavevmode\hbox to3em{\hrulefill}\thinspace}
\providecommand{\MR}{\relax\ifhmode\unskip\space\fi MR }
\providecommand{\MRhref}[2]{%
  \href{http://www.ams.org/mathscinet-getitem?mr=#1}{#2}
}
\providecommand{\href}[2]{#2}
\begin{thebibliography}{DSGG{\etalchar{+}}19}

\bibitem[AGM02]{AGM1}
Avner Ash, Paul~E. Gunnells, and Mark McConnell, \emph{Cohomology of congruence
  subgroups of {${\rm SL}_4(\Bbb Z)$}}, J. Number Theory \textbf{94} (2002),
  no.~1, 181--212. \MR{1904968}

\bibitem[AGM08]{AGM2}
\bysame, \emph{Cohomology of congruence subgroups of {${\rm SL}(4,\Bbb Z)$}.
  {II}}, J. Number Theory \textbf{128} (2008), no.~8, 2263--2274. \MR{2394820}

\bibitem[AGM10]{AGM3}
\bysame, \emph{Cohomology of congruence subgroups of {${\rm SL}_4(\Bbb Z)$}.
  {III}}, Math. Comp. \textbf{79} (2010), no.~271, 1811--1831. \MR{2630015}

\bibitem[APS18]{APS}
Avner Ash, Andrew Putman, and Steven~V. Sam, \emph{Homological vanishing for
  the {S}teinberg representation}, Compos. Math. \textbf{154} (2018), no.~6,
  1111--1130. \MR{3797603}

\bibitem[AR79]{AR}
Avner Ash and Lee Rudolph, \emph{The modular symbol and continued fractions in
  higher dimensions}, Invent. Math. \textbf{55} (1979), no.~3, 241--250.
  \MR{553998}

\bibitem[Ash94]{Ash}
Avner Ash, \emph{Unstable cohomology of {$SL(n,O)$}}, J. Algebra \textbf{167}
  (1994), no.~2, 330--342. \MR{1283290}

\bibitem[BC07]{koszul-not-cobar}
Maurizio Brunetti and Adriana Ciampella, \emph{A {P}riddy-type {K}oszulness
  criterion for non-locally finite algebras}, Colloq. Math. \textbf{109}
  (2007), no.~2, 179--192. \MR{2318516}

\bibitem[Bor74]{Bor}
Armand Borel, \emph{Stable real cohomology of arithmetic groups}, Ann. Sci.
  \'Ecole Norm. Sup. (4) \textbf{7} (1974), 235--272 (1975). \MR{0387496}

\bibitem[Bro89]{Brown-buildings}
Kenneth~S. Brown, \emph{Buildings}, Springer-Verlag, New York, 1989.
  \MR{969123}

\bibitem[BS73]{BoSe}
A.~Borel and J.-P. Serre, \emph{Corners and arithmetic groups}, Comment. Math.
  Helv. \textbf{48} (1973), 436--491, Avec un appendice: Arrondissement des
  vari\'et\'es \`a coins, par A. Douady et L. H\'erault. \MR{0387495}

\bibitem[Byk03]{Byk}
V.~A. Bykovski\u\i, \emph{Generating elements of the annihilating ideal for
  modular symbols}, Funktsional. Anal. i Prilozhen. \textbf{37} (2003), no.~4,
  27--38, 95. \MR{2083229}

\bibitem[Cal15]{calegari}
Frank Calegari, \emph{The stable homology of congruence subgroups}, Geom.
  Topol. \textbf{19} (2015), no.~6, 3149--3191. \MR{3447101}

\bibitem[CE17]{CE}
Thomas Church and Jordan Ellenberg, \emph{Homology of {FI}--modules}, Geometry
  \& Topology \textbf{21} (2017), no.~4, 2373--2418.

\bibitem[CEF15]{CEF}
Thomas Church, Jordan~S. Ellenberg, and Benson Farb, \emph{F{I}--modules and
  stability for representations of symmetric groups}, Duke Math. J.
  \textbf{164} (2015), no.~9, 1833--1910. \MR{3357185}

\bibitem[CEFN14]{CEFN}
Thomas Church, Jordan~S. Ellenberg, Benson Farb, and Rohit Nagpal,
  \emph{F{I}--modules over {N}oetherian rings}, Geom. Topol. \textbf{18}
  (2014), no.~5, 2951--2984. \MR{3285226}

\bibitem[CFP14]{CFPconj}
Thomas Church, Benson Farb, and Andrew Putman, \emph{A stability conjecture for
  the unstable cohomology of {${\rm SL}_n\Bbb Z$}, mapping class groups, and
  {${\rm Aut}(F_n)$}}, Algebraic topology: applications and new directions,
  Contemp. Math., vol. 620, Amer. Math. Soc., Providence, RI, 2014, pp.~55--70.
  \MR{3290086}

\bibitem[CFP19]{CFPint}
\bysame, \emph{Integrality in the {S}teinberg module and the top-dimensional
  cohomology of {${\rm SL}_n \mathcal O_K$}}, Amer. J. Math. \textbf{141}
  (2019), no.~5, 1375--1419. \MR{4011804}

\bibitem[Cha84]{Char}
Ruth Charney, \emph{On the problem of homology stability for congruence
  subgroups}, Comm. Algebra \textbf{12} (1984), no.~17-18, 2081--2123.
  \MR{747219}

\bibitem[CMNR18]{CMNR}
Thomas Church, Jeremy Miller, Rohit Nagpal, and Jens Reinhold, \emph{Linear and
  quadratic ranges in representation stability}, Advances in Mathematics
  \textbf{333} (2018), 1 -- 40.

\bibitem[CP17]{CP}
Thomas Church and Andrew Putman, \emph{The codimension-one cohomology of {${\rm
  SL}_n\Bbb Z$}}, Geom. Topol. \textbf{21} (2017), no.~2, 999--1032.
  \MR{3626596}

\bibitem[CS16]{CS}
Marc Chardin and Peter Symonds, \emph{Degree bounds on homology and a
  conjecture of {D}erksen}, Compos. Math. \textbf{152} (2016), no.~10,
  2041--2049. \MR{3569999}

\bibitem[Dja]{Dj}
Aur{\'e}lien Djament, \emph{De l'homologie stable des groupes de congruence},
  preprint, \url{https://hal.archives-ouvertes.fr/hal-01565891v2}.

\bibitem[DSGG{\etalchar{+}}19]{SGGHSY}
Mathieu Dutour~Sikiri\'{c}, Herbert Gangl, Paul~E. Gunnells, Jonathan Hanke,
  Achill Sch\"{u}rmann, and Dan Yasaki, \emph{On the topological computation of
  {$K_4$} of the {G}aussian and {E}isenstein integers}, J. Homotopy Relat.
  Struct. \textbf{14} (2019), no.~1, 281--291. \MR{3913976}

\bibitem[GKRW18]{Ekalgebra}
Soren Galatius, Alexander Kupers, and Oscar Randal-Williams, \emph{Cellular
  ${E}_k$-algebras}, preprint (2018), \url{http://arxiv.org/abs/1805.07184}.

\bibitem[GL19]{GLcong}
Wee~Liang Gan and Liping Li, \emph{Linear stable range for homology of
  congruence subgroups via {FI}-modules}, Selecta Math. (N.S.) \textbf{25}
  (2019), no.~4, Art. 55, 11. \MR{3997138}

\bibitem[Hau78]{MR483534}
Jean-Claude Hausmann, \emph{Manifolds with a given homology and fundamental
  group}, Comment. Math. Helv. \textbf{53} (1978), no.~1, 113--134. \MR{483534}

\bibitem[LS76a]{LSK3}
Ronnie Lee and R.~H. Szczarba, \emph{The group {$K_{3}(Z)$} is cyclic of order
  forty-eight}, Ann. of Math. (2) \textbf{104} (1976), no.~1, 31--60.
  \MR{0442934}

\bibitem[LS76b]{LS}
\bysame, \emph{On the homology and cohomology of congruence subgroups}, Invent.
  Math. \textbf{33} (1976), no.~1, 15--53. \MR{0422498}

\bibitem[LS78]{LSK45}
\bysame, \emph{On the torsion in {$K_{4}({\bf Z})$} and {$K_{5}({\bf Z})$}},
  Duke Math. J. \textbf{45} (1978), no.~1, 101--129. \MR{0491893}

\bibitem[MPP]{MPP}
Jeremy Miller, Peter Patzt, and Andrew Putman, \emph{On the top dimensional
  cohomology groups of congruence subgroups of {$\mr SL_n(\mathbb Z)$}},
  preprint, \url{https://arxiv.org/abs/1909.02661}.

\bibitem[MPW19]{MPW}
Jeremy Miller, Peter Patzt, and Jennifer C.~H. Wilson, \emph{Central stability
  for the homology of congruence subgroups and the second homology of {T}orelli
  groups}, Adv. Math. \textbf{354} (2019), 106740, 45. \MR{3992366}

\bibitem[MPWY]{MPWY}
Jeremy Miller, Peter Patzt, Jennifer C.~H. Wilson, and Dan Yasaki,
  \emph{Non-integrality of some steinberg modules}, Preprint, to appear in J.
  Topol., \url{https://arxiv.org/abs/1810.07683}.

\bibitem[Nag19]{vimod}
Rohit Nagpal, \emph{V{I}-modules in nondescribing characteristic, part {I}},
  Algebra Number Theory \textbf{13} (2019), no.~9, 2151--2189. \MR{4039499}

\bibitem[Par97]{Par}
Andrei Paraschivescu, \emph{On a generalization of the double coset formula},
  Duke Math. J. \textbf{89} (1997), no.~1, 1--8. \MR{1458968}

\bibitem[Pat]{Pa2}
Peter Patzt, \emph{Central stability homology}, preprint, to appear in Math Z.,
  \url{https://arxiv.org/abs/1704.04128v2}.

\bibitem[Pri70]{priddy-koszul}
Stewart~B. Priddy, \emph{Koszul resolutions}, Trans. Amer. Math. Soc.
  \textbf{152} (1970), 39--60. \MR{0265437}

\bibitem[PS]{PStu}
Andrew Putman and Daniel Studenmund, \emph{The dualizing module and
  top-dimensional cohomology group of {$GL_n(\mathcal O)$ }}, preprint,
  \url{https://arxiv.org/abs/1909.01217}.

\bibitem[PS17]{PS}
Andrew Putman and Steven~V. Sam, \emph{Representation stability and finite
  linear groups}, Duke Math. J. \textbf{166} (2017), no.~13, 2521--2598.
  \MR{3703435}

\bibitem[Put15]{PU}
Andrew Putman, \emph{Stability in the homology of congruence subgroups},
  Invent. Math. \textbf{202} (2015), no.~3, 987--1027. \MR{3425385}

\bibitem[Qui73]{Q}
Daniel Quillen, \emph{Finite generation of the groups {$K_{i}$} of rings of
  algebraic integers}, Algebraic {$K$}-theory, {I}: {H}igher {$K$}-theories
  ({P}roc. {C}onf., {B}attelle {M}emorial {I}nst., {S}eattle, {W}ash., 1972),
  Springer, Berlin, 1973, pp.~179--198. Lecture Notes in Math., Vol. 341.
  \MR{0349812}

\bibitem[SEVKM]{DSEVKM}
Mathieu~Dutour Sikiri{\'c}, Philippe Elbaz-Vincent, Alexander Kupers, and
  Jacques Martinet, \emph{Voronoi complexes in higher dimensions, cohomology of
  {$GL_N(Z)$} for {$N \geq 8$} and the triviality of {$K_8(Z)$}}, preprint.

\bibitem[SS15]{SS1}
Steven~V. Sam and Andrew Snowden, \emph{Stability patterns in representation
  theory}, Forum Math. Sigma \textbf{3} (2015), e11, 108. \MR{3376738}

\end{thebibliography}

%\begin{thebibliography}{CEFN}
%%	\small
%
%
%
%\bibitem[APS]{steinberg-ash-putman-sam}	Avner Ash, Andrew Putman, Steven V Sam. 
%Homological vanishing for the Steinberg representation. \arxiv{1704.08344}. 
%
%
%\end{thebibliography}

\end{document}